\documentclass[11pt]{article}

\usepackage{amsmath}
\usepackage{amsthm}
\usepackage{amssymb}
\usepackage{delarray}
\usepackage[dvips]{graphics}
\usepackage{epsfig}
\usepackage{color}
\usepackage{fullpage}
\usepackage{verbatim}

\newtheorem{theorem}{Theorem}[section]
\newtheorem{lemma}[theorem]{Lemma}
\newtheorem{proposition}[theorem]{Proposition}

\newtheorem{corollary}[theorem]{Corollary}
\newtheorem{conjecture}[theorem]{Conjecture}

\theoremstyle{definition}
\newtheorem{definition}[theorem]{Definition}
\theoremstyle{remark} 
\newtheorem{remark}[theorem]{Remark}
\theoremstyle{remark}

\newcommand\nc\newcommand
\newcommand\pDqbounded{$p$-bounded}
\newcommand\pDqbdr{$p$-bounded of exponent $r$}
\nc\gap{generalized arithmetic progression}
\nc\sgap{symmetric generalized arithmetic progression}
\nc\psgap{proper symmetric generalized arithmetic progression}
\nc\nonT{n^{o(n)}T^{O(1)}}
\nc\nonS{n^{o(n)}\lt(S \cup \{-1,0,1\}\rt)^{O(1)}}
\newcommand\gammasupermu{\gamma^{(\mu)}}

\newcommand\betasupermu{\beta^{(\mu)}}
\newcommand\alsmu{\alpha^{(\mu)}}
\newcommand\betasuper[1]{\beta^{(#1)}}
\newcommand\bmu{\beta^{(\mu)}}
\newcommand\Msn{M_{\pm 1,n}}
\newcommand\Msubnsupermu{M_{\pm 1,n}^{(\mu)}}
\newcommand\Mkn{M_{k,n}}
\newcommand\gensmu{M_{\{\pm 2,\pm 1\},n}^{(\mu)}}

\newcommand\bsubl{b_\ell}
\newcommand\bsubln[1]{b_{#1}}
\newcommand\bsubls{b_{s}}

\nc\bjsj{\bjs{j}}  
\nc\bjs[1]{b_{#1,1}}
\newcommand\psubl{p_\ell}

\newcommand\psub[1]{p_{#1}}
\newcommand\psubls{p_{s}}
\newcommand\generalmatrix{N_n}
\newcommand\glm{N_n}
\nc\fglm{\lt.\glm\rule{0pt}{11pt}\rt|_{\{\eta_{ijk} =c_{ijk}\}}}
\newcommand\glmu{\generalmatrix^{(\mu)}}
\nc\fglmu{\lt.\glmu\rule{0pt}{11pt}\rt|_{\{\eta_{ij} =c_{ij}\}}}
\newcommand\fmat{N_{\fxt,n}}
\nc\fmatkcol{\lt.\fmat\rt|_{j_1,\ldots, j_k}}
\nc\xij{x_{ij}}
\newcommand\numbervalues{T}  
\newcommand\setofvalues{S}
\newcommand\sov{\setofvalues}

\nc\hsr{highly $T$-rational}
\nc\ess{economically $T$-spans}
\newcommand\GAP{generalized arithmetic progression}
\newcommand\eQ[1]{e_Q(#1)}
\newcommand\Zstar{Z^*}
\newcommand\Ztilde{Z}
\nc\imax{{i_{\mathrm{max}}}}
\nc\iex{\imax}
\nc\Zex[1]{\Zstar_{i_{\mathrm{max}},#1}}
\nc\Xex{X_{i_{\mathrm{max}}}}
\nc\Znorm[1]{\left\| #1 \right\|_{\bb R/\bb Z}}
\nc\Pnorm[1]{\left\| #1 \right\|_{P}}
\nc\Lnorm[1]{\left\| #1 \right\|_{\Lambda}}
\nc\atil{\tilde{a}}
\nc\wtil{\tilde w}
\newcommand{\dpm}{d_{\pm}}
\newcommand{\Gr}{\operatorname{Gr}}
\newcommand{\Grf}{\operatorname{Gr}_{\fxt}}
\nc\epmo{\epsilon_{-1}}

\newcommand\mubar{\underline{\mu}}
\nc\kbar{\underline{k}}
\newcommand\Zmodp{\mathbb{Z}/Q\mathbb Z}
\newcommand\quotientmap{\phi_Q}
\newcommand\qmap{\quotientmap}

\newcommand\LOEconstant{c_{\mathrm{LO}}}
\newcommand\largedimconstant{c_{\mathrm{LgDim}}}
\newcommand\mediumdimconstant{c_{\mathrm{MedDim}}}
\newcommand\mediumdimconstantf{c_{\mathrm{MedDim},\fxt}}
\nc\mconst{c_m}
\nc\numfixed{\mathfrak{f}}
\nc\fxt\numfixed
\nc\fconst{c_\fxt}
\nc\rnk{\mathfrak{r}}
\nc\cfour{c}

\nc\dsp\displaystyle
\nc\reason[1]{\mbox{\parbox{.25\linewidth}{(#1)}}}
\nc\smlreason[2]{\mbox{\parbox{#1}{\small (#2)}}}
\nc\npreason[1]{\mbox{(#1)}}
\newcommand\probability{\Pr}
\newcommand\expectation{\mathbb{E}}
\newcommand\littleo[1]{o(#1)}
\newcommand\maximum{\max}

\newcommand\emphasis{\emph}
\newcommand\thereals{\mathbb{R}}
\newcommand\fraction[2]{\frac{#1}{#2}}
\newcommand\quantityfraction[2]{\left(\fraction{#1}{#2}\right)}
\nc\pfrac\quantityfraction
\newcommand\abs[1]{\left|#1\right|}
\newcommand\floor[1]{\left\lfloor #1 \right\rfloor}
\newcommand\ceiling[1]{\left\lceil #1 \right\rceil}
\newcommand\lessthanorequalto{\le}

\newcommand\squareroot{\sqrt}

\newcommand\MC{\mathcal}

\newcommand{\Z}{\mathbb{Z}}

\nc\rank{\operatorname{rank}}
\nc\bb\mathbb
\nc\lt\left
\nc\rt\right
\newcommand\e[1]{\begin{align*} #1
\end{align*}}
\newcommand\en[1]{\begin{align} #1
\end{align}}
\newcommand\mc\mathcal
\newcommand\fk\mathfrak

\title{On the singularity probability of discrete random matrices
}
\author{
\parbox{6in}{
	    \begin{center}
	    Jean Bourgain \\
	    \small Institute for Advanced Study, 1 Einstein dr., Princeton NJ
		08540 USA
	    \\ \texttt{bourgain@ias.edu}
	    \end{center}
	}
	\\
	\parbox{6in}{
	    \begin{center}
	    Van H. Vu\thanks{V. Vu is partly supported by NSF Career Grant 0635606
		and by an AFORS grant.}\\
	    \small Department of Mathematics, Rutgers University,
		Piscataway, NJ 08854, USA
	    \\ \texttt{vanvu@math.rutgers.edu}
	    \end{center}
	}
	\\
	\parbox{6in}{
	    \begin{center}
	    Philip Matchett Wood \\
	    \small Department of Mathematics, Rutgers University,
		Piscataway, NJ 08854, USA
	    \\ \texttt{matchett@math.rutgers.edu}
	    \end{center}
	}
}

\begin{document}
\maketitle
\begin{abstract}
Let $n$ be a large integer and $M_n$ be an $n$ by $n$ complex matrix whose
entries are independent (but not necessarily identically distributed) discrete random
variables. The main goal of this paper is to prove a general upper bound for the
probability that $M_{n}$ is singular. 

For a constant $0< p< 1$ and a constant positive integer $r$, we will define a
property \emph{\pDqbdr}. 
Our main result shows that if the entries of $M_n$ satisfy this property, then
the probability that $M_n$ is singular is at most $\lt(p^{1/r} + o(1)\rt)^n$.
All of the results in this paper hold for any characteristic zero integral
domain replacing the complex numbers.

In the special case where the entries of $M_n$ are  ``fair coin flips''
(taking the values $+1,-1$ each with probability $1/2$), our general bound
implies that the probability that $M_n$ is singular is at most
$\lt(\frac{1}{\sqrt2}+o(1)\rt)^n$, improving on the previous best upper bound
of $\lt(\frac{3}{4}+o(1)\rt)^n$, proved by Tao and Vu \cite{TV2}.

In the special case where the entries of $M_n$ are  ``lazy coin flips''
(taking values $+1,-1$ each with probability $1/4$ and value 0 with probability $1/2$), our general bound implies  that
the probability that $M_n$ is singular is at most
$\lt(\frac{1}{2}+o(1)\rt)^n$, which is asymptotically sharp.

Our method is a refinement of those  from \cite{KKS} and \cite{TV2}. In
particular,  we  make a critical use of the Structure Theorem from
\cite{TV2}, which was obtained using tools from additive combinatorics. 

  \end{abstract}

\section {Introduction}\label{s:intro}

Let $n$ be a large integer and $M_n$ be an $n$ by $n$ random matrix whose
entries are independent (but not necessarily identically distributed) discrete
random variables taking values in the complex numbers. The problem of
estimating the probability that $M_{n}$ is singular is a basic problem in the
theory of random matrices and combinatorics.  The goal of this paper is to
give a bound that applies to a large variety of distributions. The general
statement (Theorem~\ref{main theorem}) is a bit technical, so we will first
discuss a few corollaries concerning special cases.

The most famous special  case is when the entries of $M_{n}$ are
independent identically distributed (i.i.d.) Bernoulli random variables
(taking values $\pm 1$ with probability $1/2$). The following conjecture has
been open for quite some time:

\begin{conjecture}\label {big conjecture}
For $\Msn$ an $n$ by $n$ matrix with each entry an i.i.d.\ Bernoulli random
variable taking the values $+1$ and $-1$ each with probability $1/2 $,
\begin {equation*}
\probability (\Msn \mbox { is singular})= 
\left(\frac{1}{2}+o(1)\right)^n.
\end {equation*}
\end{conjecture}
It is easy to verify that the singularity probability is at least $(1/2)^n$ by
considering the probability that there are two equal rows (or columns). 

Even in the case of i.i.d.~Bernoulli random variables, proving that the
singularity probability is $o(1)$ is not trivial. It was first done by Koml\'os
in 1967 \cite{Kom} (see also \cite{Komlosnet}; 
\cite{Slinko} generalizes Koml\'os's bound to other integer
distributions).  The first exponential bound was proven by Kahn, Koml\'os, and
Szemer\'edi \cite{KKS}, who showed that $\probability (\Msn \mbox { is
singular}) \le .999^n$.  This upper bound was improved upon by Tao
and Vu in \cite{TV1} to $.958^n$. A more significant improvement was obtained by
the same authors in \cite{TV2}: 

\begin {equation}\label {best so far}
\probability (\Msn \mbox { is singular}) \le
\left(\frac{3}{4}+o(1)\right)^n.
\end {equation}

This improvement was made possible through the discovery of a new theorem
\cite[Theorem~5.2]{TV2} (which was called the Structure Theorem in
\cite{TV2}), which gives a complete characterization of a set with certain
additive properties. The Structure Theorem (to be more precise, a variant of
it) will play a critical role in the current paper as well. 

Our general result has the following corollary in the Bernoulli case:

\begin {equation}\label {very best so far}
\probability (\Msn \mbox { is singular}) \le
\left(\frac{1}{\sqrt{2}}+o(1)\right)^n,
\end {equation}
which gives a slight improvement over Inequality~\eqref{best so far} (since
$1/\sqrt{2} \approx 0.7071 < .75$).

Let us now discuss  a more general class of random matrices. 
Consider the random variable
$\gammasupermu$ defined by

\begin{equation}\label{definition of gammasupermu}
\gammasupermu := \begin{cases}
+1 & \mbox { with probability } \mu/2 \\
0 & \mbox { with probability } 1-\mu \\
-1 & \mbox { with probability } \mu/2,
\end{cases}
\end{equation}
and let $\Msubnsupermu$ be an $n$ by $n$ matrix with each entry an independent
copy of $\gammasupermu$.
The random variable $\gammasupermu$ plays an important  role in \cite{KKS,TV1,TV2}, and the matrices $\Msubnsupermu$ are of interest in their own
right.  In fact, giving zero a large weight is a natural thing to do when one
would like to (randomly) sparsify a matrix, a common operation used in
randomized algorithms (the values of $\pm 1$, as the reader will see, are not
so critical).  Our general result implies the following upper bounds:

\begin{align}
\probability (\Msubnsupermu \mbox{ is singular}) & \le (1-\mu + o(1)) ^n 
	& \mbox { for } 0\le \mu\le\frac12  \label{result 1}\\
\probability (\Msubnsupermu \mbox{ is singular}) & \le \left(\frac{2\mu +1}{4} + o(1)\right) ^n 
	& \mbox { for } \frac 12\le \mu\le 1 \label{result 2} \\
\probability (\Msubnsupermu \mbox{ is singular}) & \le 
	\left( \squareroot{1-2\mu +\frac32\mu^2} + o(1)\right)^n & \mbox{ for }
	0 \le \mu \le 1.
	\label{result 3}
\end{align}
Note that Inequality~\eqref{result 2} implies Inequality~\eqref{best so far} and
that Inequality~\eqref{result 3} implies Inequality~\eqref{very best so far} (in
both cases setting $\mu = 1 $).  

Figure~\ref{figure 1} summarizes the upper bounds from
Inequalities~\eqref{result 1}, \eqref{result 2}, and \eqref{result 3} and
also includes the following lower bounds: 
\begin{align}
(1-\mu+o(1))^n &\le \probability (\Msubnsupermu \mbox { is singular}) 
	&\mbox { for } 0\lessthanorequalto \mu \lessthanorequalto 
	1\, \, \label {lower bound 1}\\
\left(1-2\mu +\fraction 32 \mu ^ 2+\littleo {1}\right) ^n
	&\lessthanorequalto \probability (\Msubnsupermu \mbox { is singular}) 
	&\mbox { for } 
	0 \lessthanorequalto \mu \lessthanorequalto 1.  \label {lower bound 2}
\end{align}
These lower bounds can be derived by computing the probability that one row
is all zeros (Inequality~\eqref{lower bound 1}) or that there is a
dependency between two rows (Inequality~\eqref{lower bound 2}).  
Note that in the case where $\mu \le 1/2$, the upper bound in
Inequality~\eqref{result 1} asymptotically equals the lower bound in
Inequality~\eqref{lower bound 1}, and thus our result is the best possible in
this case.
%
We also used a Maple program 
to derive the formulas for lower bounds resulting from a dependency between
three, four, or five rows; however, these lower bounds were inferior to those
in Inequality~\eqref{lower bound 1} and Inequality~\eqref{lower bound 2}. 

\begin{figure}
\centerline {\textbf{Asymptotic Upper and Lower Bounds for
$\displaystyle\probability \lt(\Msubnsupermu \mbox { is singular}\rt)^{1/n} $
for $0 \le \mu \le 1 $}}
\begin{center}
\begin{picture}(0,0)%
\epsfig{file=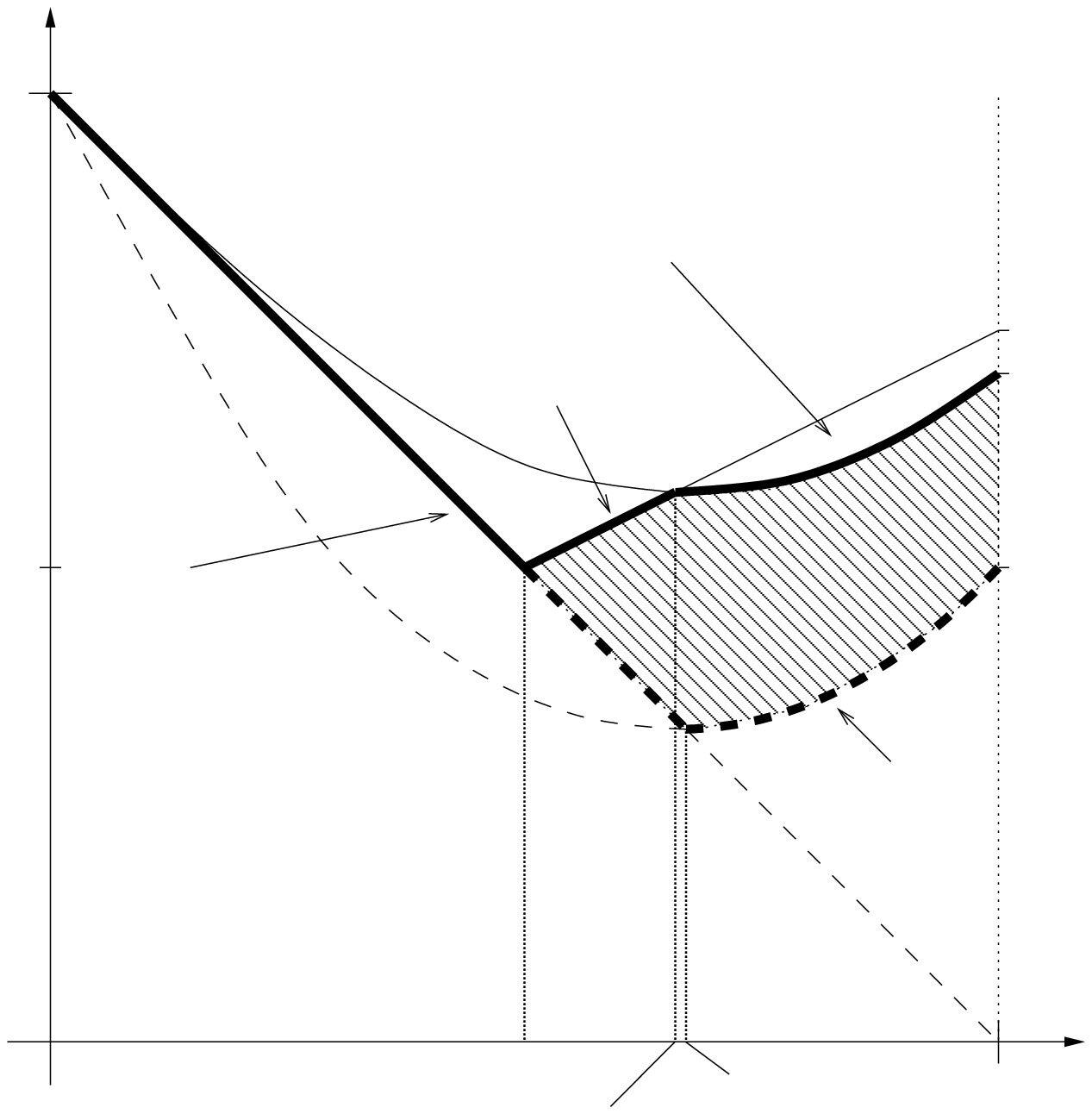}%
\end{picture}%
\setlength{\unitlength}{3158sp}%
\begingroup\makeatletter\ifx\SetFigFont\undefined%
\gdef\SetFigFont#1#2#3#4#5{%
  \reset@font\fontsize{#1}{#2pt}%
  \fontfamily{#3}\fontseries{#4}\fontshape{#5}%
  \selectfont}%
\fi\endgroup%
\begin{picture}(7725,8225)(151,-2294)
\put(4126,-2236){\makebox(0,0)[lb]{\smash{{\SetFigFont{10}{12.0}{\familydefault}{\mddefault}{\updefault}{\color[rgb]{0,0,0}$\frac{9-\sqrt{6}}{10}\approx 0.655$}%
}}}}
\put(447,5775){\makebox(0,0)[lb]{\smash{{\SetFigFont{10}{12.0}{\familydefault}{\mddefault}{\updefault}{\color[rgb]{0,0,0}$P(\mu)$}%
}}}}
\put(5776,164){\makebox(0,0)[lb]{\smash{{\SetFigFont{10}{12.0}{\familydefault}{\mddefault}{\updefault}{\color[rgb]{0,0,0}$1-2\mu+\frac{3}2 \mu^2$}%
}}}}
\put(5326,-1861){\makebox(0,0)[lb]{\smash{{\SetFigFont{10}{12.0}{\familydefault}{\mddefault}{\updefault}{\color[rgb]{0,0,0}$\displaystyle\frac23$}%
}}}}
\put(7126,-1861){\makebox(0,0)[lb]{\smash{{\SetFigFont{10}{12.0}{\familydefault}{\mddefault}{\updefault}{\color[rgb]{0,0,0}1}%
}}}}
\put(676,-1786){\makebox(0,0)[lb]{\smash{{\SetFigFont{10}{12.0}{\familydefault}{\mddefault}{\updefault}{\color[rgb]{0,0,0}0}%
}}}}
\put(226,4964){\makebox(0,0)[lb]{\smash{{\SetFigFont{10}{12.0}{\familydefault}{\mddefault}{\updefault}{\color[rgb]{0,0,0}1}%
}}}}
\put(151,1739){\makebox(0,0)[lb]{\smash{{\SetFigFont{10}{12.0}{\familydefault}{\mddefault}{\updefault}{\color[rgb]{0,0,0}$\displaystyle\frac12$}%
}}}}
\put(7351,2939){\makebox(0,0)[lb]{\smash{{\SetFigFont{10}{12.0}{\familydefault}{\mddefault}{\updefault}{\color[rgb]{0,0,0}$\displaystyle \frac{1}{\sqrt{2}}$}%
}}}}
\put(7876,-1636){\makebox(0,0)[lb]{\smash{{\SetFigFont{10}{12.0}{\familydefault}{\mddefault}{\updefault}{\color[rgb]{0,0,0}$\mu$}%
}}}}
\put(7351,1664){\makebox(0,0)[lb]{\smash{{\SetFigFont{10}{12.0}{\familydefault}{\mddefault}{\updefault}{\color[rgb]{0,0,0}$\displaystyle \frac12$}%
}}}}
\put(3676,2939){\makebox(0,0)[lb]{\smash{{\SetFigFont{10}{12.0}{\familydefault}{\mddefault}{\updefault}{\color[rgb]{0,0,0}$\frac{2\mu+1}{4}$}%
}}}}
\put(7351,3464){\makebox(0,0)[lb]{\smash{{\SetFigFont{10}{12.0}{\familydefault}{\mddefault}{\updefault}{\color[rgb]{0,0,0}$\displaystyle \frac34$}%
}}}}
\put(3826,-1861){\makebox(0,0)[lb]{\smash{{\SetFigFont{10}{12.0}{\familydefault}{\mddefault}{\updefault}{\color[rgb]{0,0,0}$\displaystyle\frac12$}%
}}}}
\put(1051,1514){\makebox(0,0)[lb]{\smash{{\SetFigFont{10}{12.0}{\familydefault}{\mddefault}{\updefault}{\color[rgb]{0,0,0}$1-\mu$}%
}}}}
\put(4276,3989){\makebox(0,0)[lb]{\smash{{\SetFigFont{10}{12.0}{\familydefault}{\mddefault}{\updefault}{\color[rgb]{0,0,0}$\sqrt{1-2\mu+\frac32 \mu^2}$}%
}}}}
\end{picture}%
\end{center}
\caption{
Let $\dsp P(\mu):= \lim_{n\to\infty}\probability \lt(\Msubnsupermu \mbox { is
singular}\rt)^{1/n} $, where $\Msubnsupermu $ is the $n$ by $n$ matrix with
independent random entries taking the value $0$ with probability $1-\mu$
and the values $+1$ and $-1$ each with probability $\mu/2$.
The solid lines denote the upper bounds on $P(\mu)$ given by
Inequalities~\eqref{result 1}, \eqref{result 2}, and \eqref{result 3}, and
the dashed lines denote the lower bounds given by 
Inequalities~\eqref {lower bound 1} and \eqref{lower bound 2}. The
upper and lower bounds coincide for $0 \le \mu \le \fraction12 $, and the
shaded area shows the difference between the best known upper and lower bounds
for $\frac12 \le \mu \le 1$.
The straight line segments from the point $(0,1)$ to $(1/2,1/2) $ and from the
point $(1/2,1/2) $ to $(1,3/4)$ represent the best upper bounds we have
derived using the ideas in \cite{TV2}, and the curve $1-2\mu+\frac32\mu^2$ for
$0\le\mu\le 1$ represents a sometimes-better upper bound we have derived by
adding a new idea.
Note that the upper bounds given here also apply to the singularity
probability of a random matrix with independent entries having arbitrary
symmetric distributions in a set $S$ of complex numbers, so long as each entry
is 0 with probability $1-\mu$ and the cardinality of $S$ is $\abs{S}\le O(1)$
(see Corollary~\ref{gen result}).
}
\label{figure 1}
\end{figure}

We will now present another corollary of the main theorem that has a somewhat
different flavor. In this corollary, we treat partially random matrices, which
may have many deterministic rows.  Our method allows us to obtain exponential
bounds so long as there are still at most $c\ln n$ random rows, where $c >0$
is a particular constant. 

\begin{corollary} \label{cor:partial} 
Let $p$ be a real constant between 0 and 1, let $c$ be any positive constant
less than $1/\ln(1/p)$, and let $S\subset \bb C$ be a set of complex numbers
having cardinality $\abs S\le O(1)$.  Let $\fmat$ be an $n$ by $n$ complex
matrix in which $\fxt\le c\ln n$ rows contain fixed, non-random elements of
$S$ and where the other rows contain entries that are independent random
variables taking values in $S$.  If the fixed rows are linearly independent
and if for every random entry $\alpha$, we have $\max_x \Pr(\alpha = x) \le
p$, then
$$\Pr(\fmat \mbox{ is singular})\le \lt(\sqrt p + o(1)\rt)^{n}.$$
\end{corollary}
\noindent
Notice that the case $\fxt=0$ and $p=1/2$ also implies 
Inequality~\eqref{very best so far}.

\begin{remark}[Other exponential bounds]\label{noSbound}
The focus of this paper is optimizing the base of the exponent in bounds on
the singularity probability for discrete random matrices.  One main tool in
this optimization is the use of a structure theorem similar to
\cite[Theorem~5.2]{TV2} (see Theorem~\ref{structure theorem} below); however,
using such a theorem requires additional assumptions to be placed on the
values that can appear as entries, and in particular, this is why we assume
in Corollary~\ref{cor:partial} that the set $S$ has cardinality $\abs S \le
O(1)$ and that $\fxt \le c\ln n$.
If one is interested in an exponential bound where there are no conditions on
$\fxt$ or on the set $S$ (at the expense of having an unspecified constant for
the base of the exponential), one can follow the analysis in \cite{TV1}, which
does not make use of a structure theorem, along with ideas in this paper
%
%
to get a result of the following form:
\end{remark}

\begin{theorem}\label{noSthm}
For every $\epsilon >0$ there exists $\delta>0$ such that the following holds.
Let $\fmat$ be an $n$ by $n$ complex matrix in which $\fxt$ rows contain
fixed, non-random entries and where the other rows contain entries that are
independent discrete random variables.  If the fixed rows have co-rank $k$ and
if for every random entry $\alpha$, we have $\max_x \Pr(\alpha = x) \le
1-\epsilon$, then for all sufficiently large $n$
$$\Pr(\fmat \mbox{ has co-rank } >k)\le (1-\delta)^{n-\fxt}.$$
%
%
\end{theorem}

\noindent
Note that Theorem~\ref{noSthm} holds for any $\fxt$ and $k$, and so in
particular, an exponential bound on the singularity probability is achieved
whenever $k=0$ and $\fxt \le cn$, where $c <1$ is a constant.  Also note that
the theorem allows the random entries to have discrete distributions taking
infinitely many values.  Corollary~\ref{cor:partial2} proves a version of
Theorem~\ref{noSthm} with a much better exponential bound, given some
additional conditions.

\medskip

The structure of the rest of the paper is as follows.  In
Section~\ref{s:genthm} we define \pDqbdr\ and state the main theorem of this
paper.  In Section~\ref{S:examples}, we discuss some  corollaries of
Theorem~\ref{main theorem}. In particular, we will:
\begin{enumerate}
\item[(A)]  prove  Inequalities~\eqref{result 1}, \eqref{result 2}, and
\eqref{result 3};

\item[(B)] prove general bounds on the singularity probability for discrete
random matrices with entries that have symmetric distributions and with
entries that have asymmetric distributions; 

\item[(C)] Prove a version of Corollary~\ref{cor:partial} (namely,
Corollary~\ref{cor:partial1}) that holds for up to $o(n)$ fixed rows, assuming
that the entries in the fixed rows take integer values between $-C$ and $C$
for any positive constant $C$; and

\item[(D)] prove that the probability that random matrices with
integer entries have a rational eigenvalue is exponentially small.
\end{enumerate}
\noindent
In Section~\ref{reducing from thereals}, we discuss Lemma~\ref{reduction
theorem}, a result that is proved in \cite{VWnote} using standard tools from
algebraic number theory and algebraic geometry. Lemma~\ref{reduction theorem}
reduces the question of bounding the singularity probability of a random
matrix with entries in $\bb C$ to a question of bounding the singularity
probability of a random matrix with entries in $\Zmodp$ for some large prime
$Q$ (in fact, it is possible to replace $\bb C$ with any characteristic zero
integral domain).  The proof of Theorem~\ref{main theorem}
is outlined in Section~\ref{S:pf}, where we also prove some of the easier
lemmas needed for the theorem.  In Section~\ref{S:except}, we state a
structure theorem (Theorem~\ref{structure theorem}) that completes the proof
of our Theorem~\ref{main theorem} and that is very similar to
\cite[Theorem~5.2]{TV2} (which is the Structure Theorem in 
\cite{TV2}). We discuss the proof of Theorem~\ref{structure
theorem}, which uses discrete Fourier analysis and tools from additive
combinatorics, in Sections~\ref{S:7} and \ref{S:8}. 
Finally, in Section~\ref{S:fxt} we show that the entire argument proving
Theorem~\ref{main theorem} can be generalized to random complex matrices with
$\fxt$ rows of the matrix containing fixed, non-random entries, so long as
$\fxt \le c\ln n$ for a particular constant $c>0$ (this leads to Corollary~\ref{cor:partial}).  

\section {The general theorem} \label{s:genthm}

To prove the results in Inequalities~\eqref{best so far} and \eqref{very best
so far} (and also the results in \cite{KKS} and \cite{TV1}), one basic idea
is to replace entries of a random matrix with independent copies of the
random variable $\gammasupermu$ or $2\gammasupermu$ (see
Equation~\eqref{definition of gammasupermu}).  One key idea in
proving the more general results of the current paper is replacing the entries
of a random matrix with more complicated symmetric discrete random variables.
 
A \gap\  of rank $\rnk$ is a set of the form $\{v_0 + m_1v_1 + \cdots + m_\rnk
v_\rnk: \abs{m_i} \le M_i/2\}$, where the $v_i$ are elements of a $\Z$-module
and the $m_i$ and $M_i>0$ are integers.
Note that whenever the term ``symmetric'' is used in this paper, it will apply
to the distribution of a random variable or to a
generalized arithmetic progression; in particular, the term will never apply to matrices. 
Also, throughout this paper we will use the notation $$e(x):= \exp(2\pi i x).$$

The following definition lies at the heart of our analysis.

\begin{definition}[\pDqbounded\ of exponent $r$]\label{definition pDqbounded}
Let $p$ be a positive constant such that $0< p<1$ and let $r$ 
be a positive integer constant.  A random variable $\alpha$ taking values in
the integers (or, respectively, the integers modulo some large prime $Q$) is \emph{\pDqbounded\
of exponent $r$} if 
\begin{enumerate}
\item[(i)] \quad$\max_x \Pr(\alpha = x) \le p$, and
\end{enumerate}
if there exists a constant $q$ where $0< q\le p$ and a $\bb Z$-valued (or, respectively, a $\Zmodp$-valued) symmetric
random variable $\betasupermu$ taking the value $0$ with probability $1-\mu =
p$ such that the following two
conditions hold:
\begin{enumerate}
\item[(ii)] \quad$q \le \min_x\Pr(\betasupermu =x)$ and
$\max_x\Pr(\betasupermu =x) \le p$, and
\item[(iii)] the following inequality holds for every $t \in \thereals$:
\begin {equation*}
\abs{\mathbb E(e(\alpha t))}^r \le \expectation\left(e(\betasupermu t)\right)
\end {equation*}
Here, if the values of $\alpha$ and $\betasupermu$ are in $\Zmodp$, we 
view those values as integers in the range $\dsp\lt(-Q/2, Q/2\rt)$ (note
that each element in $\Zmodp$ has a unique such integer representation).
\end{enumerate}
\end{definition}
We will define \pDqbdr\ for collections of random variables below, but
first we note that the conditions above are easy to verify in practice.  
In particular, if we have a symmetric random variable 
\en{\label{betasuper}
\betasupermu = \begin{cases}
\bsubl &\mbox{ with probability } \mu\psubl/2 \\
\vdots & \quad\quad\vdots\\
\bsubln{1} & \mbox{ with probability } \mu\psub{1}/2 \\
0 & \mbox{ with probability } 1-\mu \\
-\bsubln{1} & \mbox{ with probability } \mu\psub{1}/2 \\
\vdots & \quad\quad\vdots\\
-\bsubl &\mbox{ with probability } \mu\psubl/2,
\end {cases}
}
where $b_s \in \bb Z$ for all $s$ (or, respectively, $b_s \in \Zmodp$ for all
$s$),
then condition (iii) becomes
\begin {equation}\label{pDqbounded}
\abs{\mathbb E(e(\alpha t))}^r \le \expectation\left(e(\betasupermu t)\right)
= 1 - \mu + \mu \sum_{s=1}^\ell \psubls\cos 2\pi \bsubls t,
\end {equation}
where the equality on the right-hand side is a simple expected value
computation.

We say that a collection of random variables $\{\alpha_{jk}\}_{j,k=1}^n$ is
\pDqbounded\ of exponent $r$ if each $\alpha_{jk}$ is \pDqbounded\ of exponent
$r$ with the same constants $p$, $q$, and $r$; and, importantly, the same
value of $\mu = 1-p$. We also make the critical assumption that the set of
all values that can be taken by the $\betasupermu_{jk}$ has cardinality
$O(1)$ (a relaxation of this assumption is discussed in Remark~\ref{rlax}).  
However, the definition of $\betasupermu_{jk}$ is otherwise allowed to vary
with $j$ and $k$.  Also, we will use $\sov$ to denote the set of all possible
values taken by the random variables $\alpha_{jk}$, and we will assume that
the cardinality of $\sov$ is at most $\abs \sov \le n^{o(n)}$.

If $\alpha$ takes non-integer values in $\bb C$, we need to map those values
to a finite field of prime order so that we may use Definition~\ref{definition
pDqbounded}, and for this task we will apply Lemma~\ref{reduction theorem},
which was proved in \cite{VWnote}.  
We say that $\alpha$ is \pDqbounded\ of exponent $r$ if and only if for each
prime $Q$ in an infinite sequence of primes produced by Lemma~\ref{reduction
theorem}, we have $\qmap(\alpha)$ is $p$-bounded of exponent $r$, where
$\qmap$ is the ring homomorphism described in Lemma~\ref{reduction theorem}
that maps $\sov$, the finite set of all possible values taken by the
%
%
$\alpha_{jk}$, into $\Zmodp$ in such a way that for any matrix
$\generalmatrix:=(s_{jk})$ with entries in $\sov$, the determinant of
$\generalmatrix$ is zero if and only if the determinant of
$\qmap(\generalmatrix):=(\qmap(s_{jk}))$ is zero.

\begin{theorem}\label{main theorem}
Let $p$ be a positive constant such that $0< p<1$, let $r$ be a positive
integer constant, and let $S$ be a \gap\ in the complex numbers with rank
$O(1)$ (independent of $n$) and with cardinality at most $\abs S \le
n^{o(n)}$.
Let $\generalmatrix$ be an $n$ by $n$ matrix with entries $\alpha_{jk}$, each
of which is an independent random variable taking values in $S$. If the collection of random variables
$\{\alpha_{jk}\}_{1\le j,k\le n}$  is \pDqbounded\ of exponent $r$, then
\[
\probability (\generalmatrix \mbox { is singular})
\lessthanorequalto (p^{1/r} +\littleo{1}) ^n.
\]
\end {theorem}

In the motivating
examples of Section~\ref{s:intro} (excluding Corollary~\ref{cor:partial}), we
discussed the case where the entries of the matrix are i.i.d.; however, in
general the distributions of the entries are allowed to differ (and even
depend on $n$), so long as the entries all take values in the same structured
set $S$ described above.
The condition that $S$ has additive structure seems to be an artifact of the
proof (in particular, at certain points in the proof of Theorem~\ref{structure
theorem}, we need the set $\lt\{\sum_{j=1}^n x_j : x_j \in S \mbox{ for all }
j\rt\}$ to have cardinality at most $n^{o(n)}$).  The easiest way to guarantee
that $S$ has the required structure is to assume that the set of values taken
by all the $\alpha_{jk}$ has cardinality at most $O(1)$, and this is the
approach we take for the corollaries in Section~\ref{S:examples}, since it
also makes it easy to demonstrate that the collection of entries is \pDqbdr.

\begin{remark}[Strict positivity in Inequality~\eqref{pDqbounded}]
\label{rem:strict positivity}
Note that the constants $\mu,\psubls, \bsubls$ must be such that the right-hand side of Equation~\eqref{pDqbounded} is non-negative.  It turns out for
the proof of Theorem~\ref{main theorem} that we will need slightly more.
At one point in the proof, we will apply Lemma~\ref{fxi_lower}, for which we
we must assume that there exists a very small constant $\epmo>0$ such that $\bb
E(e(\betasupermu_{jk} t))> \epmo$  for all $t$ and for all $\betasupermu_{jk}$
used in the definition of \pDqbounded\ of exponent $r$.  Of course, if the
expectations are not strictly larger than $\epmo$, we can simply reduce $\mu$
by $\epsilon_{-1}>0$.  Then, since we are assuming $1-\mu=p$, we clearly have
that all the $\alpha_{jk}$ are $(p+\epsilon_{-1})$-bounded of exponent $r$
(by using $\betasuper{\mu-\epmo}_{jk}$ instead of $\betasupermu_{jk}$) and we
have that $\bb E(e(\betasuper{\mu-\epmo}_{jk} t))> \epmo>0$.  Since
Theorem~\ref{main theorem} would thus yield a bound of
$\lt((p+\epsilon_{-1})^{1/r} + o(1)\rt)^n$ for every $\epsilon_{-1}>0$, we can
conclude a bound of $\lt(p^{1/r} + o(1)\rt)^n$ by letting $\epmo$ tend to 0.
Thus, without loss of generality, we will assume that $\bb
E(e(\betasupermu_{jk} t))> \epmo$ for all $t$ and for all $\betasupermu_{jk}$
used in the definition of \pDqbounded\ of exponent $r$.  \end{remark}

\section{Some corollaries of Theorem~\ref{main theorem}} 
\label{S:examples}

In this section, we will state a number of corollaries of Theorem~\ref{main
theorem}, starting with short proofs of Inequalities~\eqref{result 1},
\eqref{result 2}, and \eqref{result 3}.  The two most interesting results in
this section will be more general: first (in Section~\ref{ss:sym}), we will
show an exponential bound on the singularity probability for a matrix with
independent entries
each a symmetric random variable taking values in $\sov\subset \bb
C$, where $\abs \sov\le O(1)$ and assuming that each entry takes the value 0 with probability $1-\mu$; and
second (in Section~\ref{S:gen asym}), we will describe a similar (and
sometimes better) bound when the condition that the random variables have
symmetric distributions is replaced
with the assumption that no entry takes a value with probability greater than
$p$.  In the first case, the bound will depend only the value of $\mu$, and in
the second case, the bound will depend only on the value of $p$.
In Section~\ref{ss:fxt}, we will show an exponential bound on the singularity
probability for an $n$ by $n$  matrix with $\fxt = o(n)$ fixed rows containing
small integer values and with the remaining rows containing independent random
variables taking values in $S\subset \bb C$, where $\abs S\le O(1)$ (this is
similar to Corollary~\ref{cor:partial}, which is proved in
Section~\ref{S:fxt}).  Finally, in Section~\ref{ss:eigen}, we will prove an
exponential upper bound on the probability that a random integer matrix has a
rational eigenvalue.

In each corollary, we will use the definition of \pDqbounded\ of exponent 1
and of exponent 2.  The definition of \pDqbounded\ of exponent 2 is
particularly useful, since then the absolute value on the left-hand side of
Inequality~\eqref{pDqbounded} is automatically dealt with; however, when $\mu$
is small (for example whenever $\mu \le 1/2$), one can get better bounds by
using \pDqbounded\ of exponent 1.  We have not yet found an example where the
best possible bound from Theorem~\ref{main theorem} is found by using
\pDqbounded\ of an exponent higher than 2.

\subsection{Proving Inequalities~\eqref{result 1}, \eqref{result 2}, and
\eqref{result 3}}

To prove Inequality~\eqref{result 1}, we note for 
$0 \le\mu \le \frac12$ that (using the definition in
Equation~\eqref{definition of gammasupermu} of $\gammasupermu $) 
\begin{equation*}
\abs{\mathbb E(e(\gammasupermu t))} = 1 - \mu + \mu \cos(2\pi t),
\end{equation*}
and thus $\gammasupermu$ is $(1-\mu)$-bounded of exponent 1
(i.e., take $\betasupermu:=\gammasupermu$), and so Inequality~\eqref{result 1}
follows from Theorem~\ref{main theorem}.

To prove Inequality~\eqref{result 2}, we note for $\frac12 \le \mu
\le 1$ that 
\begin{equation*}
\abs{\mathbb E(e(\gammasupermu t))} = \abs{1 - \mu + \mu \cos( 2\pi t)} \le
\pfrac{2\mu+1}{4} +(1-\mu)\cos(2\pi t) + \pfrac{2\mu-1}{4}\cos(4\pi t)
\end{equation*}
(the inequality above may be checked by squaring both sides and expanding as
polynomials in $\cos(2\pi t)$).  Thus, we can take 
\e{
	\betasupermu:= \begin{cases}
	+2 &\mbox{ with probability } \frac{2\mu-1}{8} \\
	-2 &\mbox{ with probability } \frac{2\mu-1}{8} \\
	+1 &\mbox{ with probability } \frac{1-\mu}{2} \\
	-1 &\mbox{ with probability } \frac{1-\mu}{2} \\
	0 &\mbox{ with probability } \frac{2\mu+1}{4}
	\end{cases}
}
to see that $\gammasupermu$ is $\displaystyle \left(\frac{2\mu+1}{4}
\right)$-bounded of exponent 1, and so Inequality~\eqref{result 2} follows
from Theorem~\ref{main theorem}.

To prove Inequality~\eqref{result 3}, we note for $0 \le\mu
\le 1$ that 
\begin{equation*}
\abs{\mathbb E(e(\gammasupermu t))}^2 = \abs{1 - \mu + \mu \cos( 2\pi t)}^2 
= 1-2\mu+\frac32 \mu^2 +2(1-\mu)\mu\cos(2\pi t) + \pfrac{\mu^2}{2}\cos(4\pi
t).
\end{equation*}
Thus, we can take 
\e{
	\betasupermu:= \begin{cases}
	+2 &\mbox{ with probability } \frac{\mu^2}{4} \\
	-2 &\mbox{ with probability } \frac{\mu^2}{4} \\
	+1 &\mbox{ with probability } (1-\mu)\mu \\
	-1 &\mbox{ with probability } (1-\mu)\mu \\
	0 &\mbox{ with probability }  1-2\mu+\frac32 \mu^2
	\end{cases}
}
to see that $\gammasupermu$ is $\displaystyle \left(1-2\mu+\frac32 \mu^2
\right)$-bounded of exponent 2, and so Inequality~\eqref{result 3} follows
from Theorem~\ref{main theorem}.

\subsection{Matrices with entries having symmetric distributions}\label{ss:sym}

In this subsection, we will prove a singularity bound for an $n$ by $n$ matrix
$\glmu$ for which each entry is a symmetric discrete random variable
taking the value 0 with probability $1-\mu$.

\begin{corollary}\label{gen result}
Let $\sov$ be a set of complex numbers with cardinality $\abs \sov \le O(1)$.
If $\glmu$ is an $n$ by $n$ matrix in which each entry is an independent
symmetric complex random variable taking values in $S$ and taking the
value 0 with probability $1-\mu$, then
\e{
	\Pr(\glmu \mbox{ is singular}) \le
	\begin{cases}
		(1-\mu + o(1))^n &\mbox{ for } 0\le\mu\le\frac12 \\[5pt]
		\left(\frac{2\mu+1}{4} + o(1)\rt)^n&\mbox{ for } \frac12\le\mu\le1
		\\[5pt]
		\lt(\sqrt{1-2\mu +\frac32\mu^2} +o(1)\rt)^n &\mbox{ for }0\le\mu\le 1.
	\end{cases}
}
In particular, the same upper bounds as in Inequalities~\eqref{result 1},
\eqref{result 2}, and \eqref{result 3} (which are shown in Figure~\ref{figure
1}) apply to the singularity probability for $\glmu$.
\end{corollary}

\begin{proof}
Let $\alpha_{ij}$ be an entry of $\glmu$.  Since $\alpha_{ij}$ is symmetric
and takes the value 0 with probability $1-\mu$, we may write $\alpha_{ij} =
\gammasupermu_{ij} \eta_{ij}$, where $\gammasupermu_{ij}$ is an independent
copy of $\gammasupermu$ as defined in Equation~\eqref{definition of
gammasupermu} and $\eta_{ij}$ is a random variable that shares no values
with $-\eta_{ij}$.  This description of $\alpha_{ij}$ was inspired by
\cite{BerDecomp}, and it allows us to condition on $\eta_{ij}$ and then use the
remaining randomness in $\gammasupermu_{ij}$ to get a bound on the singularity
probability.  In particular,
\e{
	\Pr(\glmu \mbox{ is singular}) &= \sum_{(c_{ij})} 
	\Pr(\glmu \mbox{ is singular} | \{\eta_{ij} =c_{ij}\}) \Pr(\{\eta_{ij}
	=c_{ij}\}),
}
where the sum runs over all $(n^2)$-tuples $(c_{ij})_{1\le i,j\le n}$ of
possible values taken by random variables $\eta_{ij}$.  Since
$\sum_{(c_{ij})}  \Pr(\{\eta_{ij} =c_{ij}\}) =1$, we can
complete the proof by proving an exponential bound on $\Pr(\glmu \mbox{ is
singular} | \{\eta_{ij} =c_{ij}\})$, and we will use Theorem~\ref{main
theorem} for this task.  

Consider the random matrix $\fglmu$, where the $i,j$ entry is the random
variable $c_{ij}\gammasupermu_{ij}$ for some constant $c_{ij}$.  Note that the
entries of $\fglmu$ take values in $S$, a set with cardinality $O(1)$, and let
$\qmap$ be the map from Lemma~\ref{reduction theorem}, which lets us pass to
the case where $\fglmu$ has entries in $\Zmodp$.  Defining $\theta_{ij}:=
2\pi\qmap(c_{ij})$, we compute
\e{
	&\hspace{-1cm} \abs{\bb E e(\qmap(c_{ij}\gammasupermu_{ij}) t)} 
	=\abs{ 1-\mu+ \mu \cos(\theta_{ij}t) }
	\\
	\qquad&\le \begin{cases}
1-\mu+ \mu \cos(\theta_{ij}t)
	&\mbox{ for } 0\le\mu\le\frac12, \\[5pt]
\frac{2\mu+1}{4} +(1-\mu)\cos(\theta_{ij}t) +
\pfrac{2\mu-1}{4}\cos(2\theta_{ij} t)
	&\mbox{ for } \frac12\le\mu\le 1, \mbox{ and } \\[5pt]
\dsp
\lt(\rule{0pt}{11pt} 
1-2\mu+\frac32 \mu^2 +2(1-\mu)\mu \cos(\theta_{ij} t) +
\frac{\mu^2}{2}\cos(2\theta_{ij} t) \rt)^{1/2}
	&\mbox{ for } 0\le\mu\le 1.
\end{cases}
}
We have thus shown that the entries of $\fglmu$ are
\e{
	&\lt(1-\mu\rt)\mbox{-bounded of exponent 1 for }
	0\le\mu\le\frac12,\\
	&\lt(\frac{2\mu+1}{4}\rt)\mbox{-bounded of exponent 1 for } 
		\frac12 \le \mu\le 1, \mbox{ and}\\	
	&\lt(1-2\mu+\frac32\mu^2\rt)\mbox{-bounded of exponent 2 for } 
		0\le \mu\le 1.	
}
Applying Theorem~\ref{main theorem} completes the proof.  \end{proof}

Corollary~\ref{gen result} is tight for $0\le\mu\le \frac12$, since the
probability of a row of all zeroes occurring is $(1-\mu +o(1))^n$; however,
for any specific case, Theorem~\ref{main theorem} can usually prove better
upper bounds than those given by Corollary~\ref{gen result}.

For example, consider the case of a matrix $\gensmu$ with each entry an
independent copy of the symmetric random variable
\e{
	\alsmu:= 
\begin{cases}
	+2 &\mbox{ with probability } \frac{\mu}{4} \\
	-2 &\mbox{ with probability } \frac{\mu}{4} \\
	+1 &\mbox{ with probability } \frac{\mu}{4} \\
	-1 &\mbox{ with probability } \frac{\mu}{4} \\
	0  &\mbox{ with probability } 1-\mu
	\end{cases}
}

\begin{corollary} \label{pm2case}
For $\gensmu$ as defined above, we have
\e{
	\Pr(\gensmu \mbox{ is singular}) \le 
	\begin{cases}
	(1-\mu + o(1) )^n &\mbox{ for } 0 \le \mu \le \frac{16}{25}\\[5pt]
	\lt(\sqrt{1-2\mu+\frac54 \mu^2} + o(1) \rt)^n &\mbox{ for } 0\le\mu\le 1.
	\end{cases}
}
\end{corollary}

\begin{proof}
By the definition of $\alsmu$ we have
\e{
	\abs{\bb E e(\alsmu t)} = 1-\mu +\frac{\mu}2 \cos(2\pi t) +\frac{\mu}2
	\cos(4\pi t),\qquad\mbox{ for } 0\le\mu\le\frac{16}{25}
}
(i.e., the right-hand side of the equation above is non-negative for such
$\mu$), which proves the first bound.  

Also, we have
\e{
	\abs{\bb E e(\alsmu t)}^2 &= 1-2\mu+\frac54\mu^2 +
	\lt(\mu-\frac{3}4\mu^2\rt) \cos(2\pi t) + \lt(\mu-\frac{7}8\mu^2\rt)
	\cos(4\pi t) \\
&\rule{3in}{0pt}+ \frac{\mu^2}4 \cos(6\pi t) + \frac{\mu^2}8 \cos(8\pi t)
}
for $ 0\le\mu\le1$, which proves the second bound.  
\end{proof}

\begin{figure}
\centerline {\textbf{Asymptotic Upper and Lower Bounds for
$\displaystyle\probability \lt(\gensmu \mbox { is singular}\rt)^{1/n} $ for $0
\le \mu \le 1 $}}
\begin{center}
\begin{picture}(0,0)%
\epsfig{file=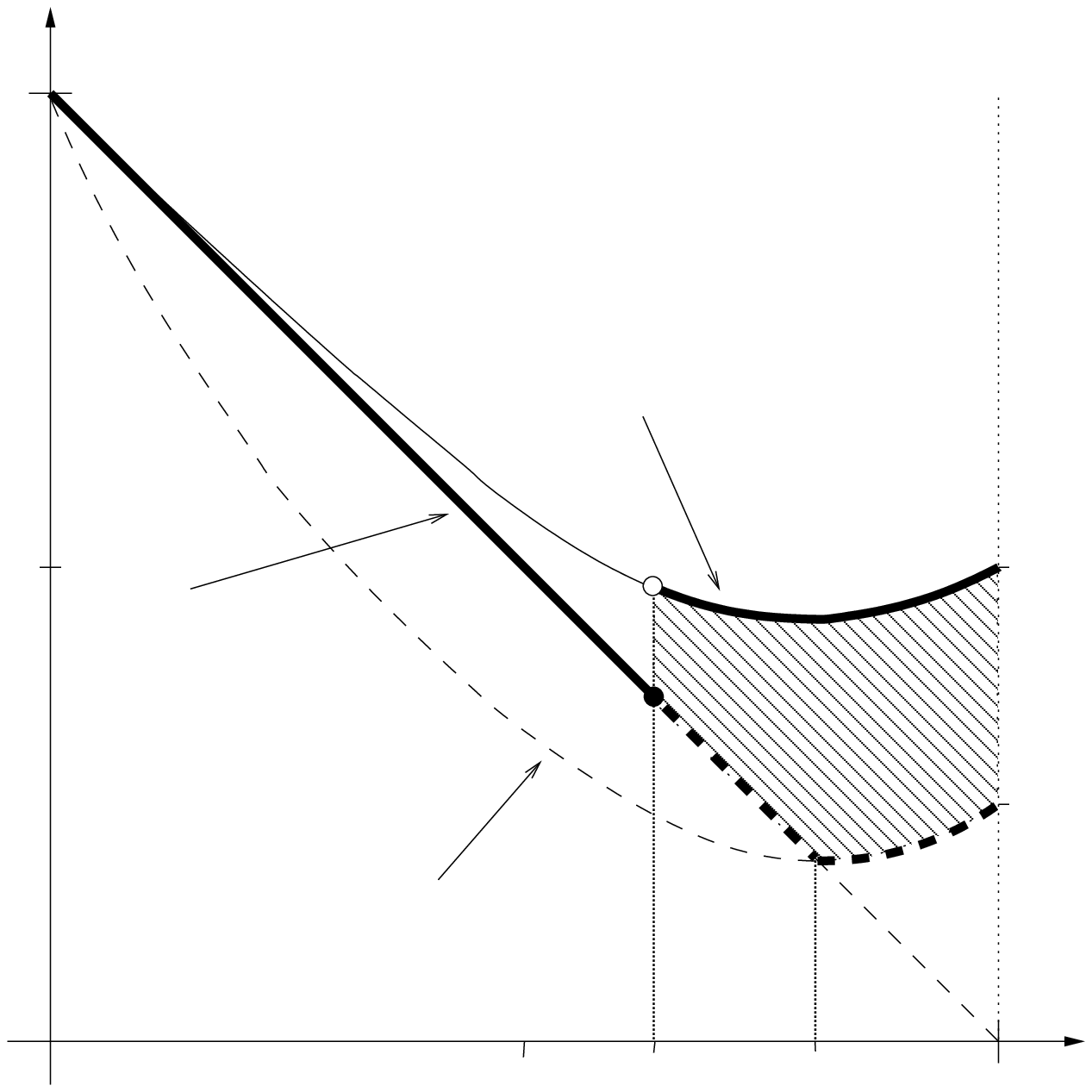}%
\end{picture}%
\setlength{\unitlength}{3158sp}%
\begingroup\makeatletter\ifx\SetFigFont\undefined%
\gdef\SetFigFont#1#2#3#4#5{%
  \reset@font\fontsize{#1}{#2pt}%
  \fontfamily{#3}\fontseries{#4}\fontshape{#5}%
  \selectfont}%
\fi\endgroup%
\begin{picture}(7725,8089)(151,-2144)
\put(451,5789){\makebox(0,0)[lb]{\smash{{\SetFigFont{10}{12.0}{\rmdefault}{\mddefault}{\updefault}{\color[rgb]{0,0,0}$P(\mu)$}%
}}}}
\put(3751,2939){\makebox(0,0)[lb]{\smash{{\SetFigFont{10}{12.0}{\familydefault}{\mddefault}{\updefault}{\color[rgb]{0,0,0}$\sqrt{1-2\mu+\frac54 \mu^2}$}%
}}}}
\put(3826,-2086){\makebox(0,0)[lb]{\smash{{\SetFigFont{10}{12.0}{\familydefault}{\mddefault}{\updefault}{\color[rgb]{0,0,0}$\displaystyle\frac12$}%
}}}}
\put(4651,-2086){\makebox(0,0)[lb]{\smash{{\SetFigFont{10}{12.0}{\familydefault}{\mddefault}{\updefault}{\color[rgb]{0,0,0}$\displaystyle \frac{16}{25}$}%
}}}}
\put(5776,-2086){\makebox(0,0)[lb]{\smash{{\SetFigFont{10}{12.0}{\familydefault}{\mddefault}{\updefault}{\color[rgb]{0,0,0}$\displaystyle\frac45$}%
}}}}
\put(7126,-1861){\makebox(0,0)[lb]{\smash{{\SetFigFont{10}{12.0}{\familydefault}{\mddefault}{\updefault}{\color[rgb]{0,0,0}1}%
}}}}
\put(676,-1786){\makebox(0,0)[lb]{\smash{{\SetFigFont{10}{12.0}{\familydefault}{\mddefault}{\updefault}{\color[rgb]{0,0,0}0}%
}}}}
\put(226,4964){\makebox(0,0)[lb]{\smash{{\SetFigFont{10}{12.0}{\familydefault}{\mddefault}{\updefault}{\color[rgb]{0,0,0}1}%
}}}}
\put(151,1739){\makebox(0,0)[lb]{\smash{{\SetFigFont{10}{12.0}{\familydefault}{\mddefault}{\updefault}{\color[rgb]{0,0,0}$\displaystyle\frac12$}%
}}}}
\put(7876,-1636){\makebox(0,0)[lb]{\smash{{\SetFigFont{10}{12.0}{\familydefault}{\mddefault}{\updefault}{\color[rgb]{0,0,0}$\mu$}%
}}}}
\put(7351,1664){\makebox(0,0)[lb]{\smash{{\SetFigFont{10}{12.0}{\familydefault}{\mddefault}{\updefault}{\color[rgb]{0,0,0}$\displaystyle \frac12$}%
}}}}
\put(1051,1514){\makebox(0,0)[lb]{\smash{{\SetFigFont{10}{12.0}{\familydefault}{\mddefault}{\updefault}{\color[rgb]{0,0,0}$1-\mu $}%
}}}}
\put(7351, 14){\makebox(0,0)[lb]{\smash{{\SetFigFont{10}{12.0}{\familydefault}{\mddefault}{\updefault}{\color[rgb]{0,0,0}$\displaystyle \frac14$}%
}}}}
\put(2764,-646){\makebox(0,0)[lb]{\smash{{\SetFigFont{10}{12.0}{\familydefault}{\mddefault}{\updefault}{\color[rgb]{0,0,0}$1-2\mu+\frac54 \mu^2$}%
}}}}
\end{picture}%
\end{center}
\caption{
	Let $P(\mu):= \lim_{n\to\infty}\probability \lt(\gensmu \mbox { is
	singular}\rt)^{1/n}$, where $\gensmu$ is the $n$ by $n$ matrix with
	independent random entries taking the value 0 with probability $1-\mu$ and
	the values $+2,-2,+1,-1$ each with probability $\mu/4$.  This figure
	summarizes the upper bounds on $P(\mu)$ from Corollary~\ref{pm2case} and
	the lower bounds from 
	Displays~\eqref{2lb1} and \eqref{2lb2}. 
	The best upper bounds (shown in thick solid lines) match the best lower
	bounds (thick dashed lines) for $0 \le\mu\le \frac{16}{25}$; and it is not
	hard to improve the upper bound a small amount by finding a bound (of
	exponent 1) to bridge the discontinuity.  One should note that even as
	stated above, the upper bounds are substantially better than those given
	by Corollary~\ref{gen result} (which are shown in Figure~\ref{figure 1}).
	The shaded area represents the gap between the upper and lower bounds.
}
\label{figure 2}
\end{figure}

We also have the following lower bounds for the singularity probability of
$\gensmu$: 
\en{
&\ (1-\mu +o(1))^n  &\npreason{from one row of all zeroes} \label{2lb1}\\
&\lt(1-2\mu+5\mu^2/4 + o(1)\rt)^n &\npreason{from a two-row
dependency}\label{2lb2}
}
The results of Corollary~\ref{pm2case} and the corresponding lower bounds are
shown in Figure~\ref{figure 2}, and one should note that the upper bounds are
substantially better than those guaranteed by Corollary~\ref{gen result}.

\subsection{Random  matrices with entries having arbitrary distributions}
\label{S:gen asym}

A useful feature of the definition of \pDqbounded\ of exponent 2 is that it
lets one bound the singularity probability of matrices with independent discrete
random variables that are asymmetric. 

\begin{corollary}\label{gen asym}
Let $p$ be a constant such that $0<p\le 1$ and let $\sov\subset \bb C$ be a set
with cardinality $\abs \sov \le O(1)$.
If $\glm$ is an $n$ by $n$ matrix with independent random entries taking
values in $S$ such that for any entry $\alpha$, we have $\max_x \Pr(\alpha =
x)\le p$, then
$$\Pr(\glm \mbox{ is singular})\le (\sqrt p + o(1))^n.$$
\end{corollary}

We will need the following slightly more general corollary in
Section~\ref{ss:fxt}. For a set $A$ and an integer $m$, we will use the
notation $mA :=\{ \sum_{j=1}^m a_j : a_j \in A\}$ and $A^m:= \{ \prod_{j=1}^m
a_j : a_j \in A\}$. 

\begin{corollary}\label{gen asym2}
Let $p$ be a constant such that $0<p\le 1$, let $\sov\subset \bb C$ be a set
with cardinality $\abs \sov \le O(1)$, and let $X_n$ be an $n$ by $n$ matrix
with fixed, non-random entries in $n^{o(n)}(S \cup \{-1,0,1\})^{O(1)}$.
If $\glm$ is an $n$ by $n$ matrix with independent random entries taking
values in $S$ such that for any entry $\alpha$, we have $\max_x \Pr(\alpha =
x)\le p$, then
$$\Pr(X_n + \glm \mbox{ is singular})\le (\sqrt p + o(1))^n.$$
\end{corollary}
\noindent
Note that that Corollary~\ref{gen asym2} implies Corollary~\ref{gen asym} by
taking $X_n$ to be the matrix of all zeroes. 

\begin{proof}[Proof of Corollary~\ref{gen asym2}]
Let $\alpha_{ij}$ be an entry in $\glm$.  Our goal is to describe
$\alpha_{ij}$ in a two-step random process, condition on one of the steps, and
then use the randomness in the other step to bound the singularity
probability.  The conditioning approach is the same as that used in the
symmetric case (Corollary~\ref{gen result}) and was inspired by
\cite{BerDecomp}.  The conditioning argument is useful since some entries of
the random matrix may take some values with very small probability (i.e.
probability less than any constant);  recall that while the entries of the
random matrix always take values in a fixed set $S$ of cardinality $O(1)$, the
distributions of those random variables within $S$ are allowed to
vary with $n$.  (Note that making use of Remark~\ref{rlax} would provide an
alternate way of dealing with entries that take some values with very small
probability.)

Say that $\alpha_{ij}$ takes the values $v_1,\ldots,v_t$
with probabilities $\varrho_1,\ldots, \varrho_t$, respectively, where
$\varrho_1\ge \varrho_2 \ge \cdots\ge \varrho_t$.  Define new random variables
$\eta_{ijk}$ such that for some $i_0$
and $i_1$, the values taken by $\eta_{ijk}$ are $v_{i_0},
v_{i_0+1},\ldots,v_{i_0 + i_1}$ with corresponding probabilities
$\varrho_{i_0}/p_k,
\varrho_{i_0+1}/p_k,\ldots,\varrho_{i_0 + i_1}/p_k$, where $p_k:=
\sum_{i=1}^{i_1}\varrho_{i_0 + i}$.  Thus, we can write
\en{ \label{couple}
	\alpha_{ij} = \begin{cases}
	\eta_{ij1} & \mbox{ with probability } p_1 \\
	\eta_{ij2} & \mbox{ with probability } p_2 \\
	\vdots & \vdots \\
	\eta_{ij\ell} & \mbox{ with probability } p_\ell.
	\end{cases}
}
Furthermore, the $\eta_{ijk}$ can be constructed so that $p_k \le p$ for
every $k$, so that $p/2 \le p_k$ for $1\le k \le \ell -1$, and so that no two
$\eta_{ijk}$ with different $k$'s ever take the same value.  

There are two cases to consider for the technical reason that $p_\ell$ is not
necessarily bounded below by a constant.  Let $\epsilon>0$ be
a very small constant, so for example $p/2> \epsilon$.  Case 1 is when $\epsilon\le p_\ell$, and in this case each
$p_k$ is bounded below by $\epsilon$ and above  by $p$.
We will consider Case 1 first and then discuss the
small changes needed to deal with Case 2.

As in the proof of Corollary~\ref{gen result}, we will
condition on the values taken by the $\eta_{ijk}$ in order to prove a
bound on the singularity probability.  We have that 
$$\Pr(X_n + \glm \mbox{ is singular}) = \sum_{(c_{ijk})} \Pr(X_n + \glm\mbox{ is
singular} | \{\eta_{ijk}=c_{ijk}\}) \Pr(\{\eta_{ijk}=c_{ijk}\}),$$
where the sum runs over all possible values $(c_{ijk})$ that the $\eta_{ijk}$
can take.  Thus, it is sufficient to prove a bound on the singularity
probability for the random matrix $X_n + \fglm$ which has random entries
\e{
	\xij+\widetilde\alpha_{ij} = \begin{cases}
	\xij+c_{ij1} & \mbox{ with probability } p_1 \\
	\xij+c_{ij2} & \mbox{ with probability } p_2 \\
	\vdots & \vdots \\
	\xij+c_{ij\ell} & \mbox{ with probability } p_\ell,
	\end{cases}
}
where $\xij$ and the $c_{ijk}$ are constants.

Note the entries of $X_n+\fglm$ take values in $\nonS$, a \gap\ with rank
$O(1)$ and cardinality at most $n^{o(n)}$, and let $\qmap$ be the
map from Lemma~\ref{reduction theorem}, which lets us pass to the case where
$X_n+\fglm$ has entries in $\Zmodp$.  Defining $\theta_{ijk}:=
2\pi\qmap(c_{ijk})$ and letting $\widetilde\alpha'_{ij}$ be an i.i.d. copy of $\widetilde\alpha_{ij}$, we compute
\e{
	\abs{\bb E e(\qmap(\xij+\widetilde\alpha_{ij}) t)}^2 
	&=\bb E e
	\lt(\rule{0pt}{11pt}
	\qmap(\xij+\widetilde\alpha_{ij}-\xij-\widetilde\alpha_{ij}') t\rt) 
	=\bb E e
	\lt(\rule{0pt}{11pt}
	\qmap(\widetilde\alpha_{ij}-\widetilde\alpha_{ij}') t\rt) 
	\\
	&= \sum_{k=1}^\ell p_k^2 + 2 \sum_{1\le k_1<k_2\le \ell} p_{k_1} p_{k_2}
	\cos((\theta_{ijk_1}-\theta_{ijk_2})t).
}
Thus, $\xij+\widetilde\alpha_{ij}$ is $\lt(\sum_{k=1}^\ell p_k^2\rt)$-bounded of
exponent 2 (using the constant $q = \epsilon^2$ in Definition~\ref{definition
pDqbounded}, so $q$ does not depend on $n$). Given that $0<p_k\le p$ for every
$k$, it is not hard to show that $\sum_{k=1}^\ell p_k^2  \le p< p+\epsilon$,
and so from Definition~\ref{definition pDqbounded}, we see that the collection
$\{\xij+\widetilde\alpha_{ij}: \widetilde\alpha_{ij}\mbox{ has corresponding
probability } p_\ell \ge \epsilon\}$ is $\lt(p + \epsilon\rt)$-bounded of
exponent 2.  We are thus finished with Case 1.

Case 2 is when the decomposition of $\alpha_{ij}$ given in Equation~\eqref{couple} has
$p_\ell < \epsilon$.  In this case we need to modify Equation~\eqref{couple}
slightly, deleting $\eta_{ij\ell}$ and replacing $\eta_{ij(\ell-1)}$ with a new
variable $\eta_{ij(\ell-1)}'$ that takes all the values previously taken by
$\eta_{ij\ell}$ and by $\eta_{ij(\ell-1)}$ with the appropriate probabilities.  
Thus, in Case 2, we have that $p/2\le p_k < p+\epsilon$ for all $1\le k\le
\ell -1$, showing that each $p_k$ is bounded below by a constant and is
bounded above by $p+\epsilon$ (here we are using $p_{\ell -1}$ to denote
the probability that $\alpha_{ij}$ draws a value from the random variable
$\eta_{ij(\ell-1)}'$).

For Case 2, we use exactly the same reasoning as in Case 1 above to show that 
such entries of $X_n+\fglm$ are $\lt(\sum_{k=1}^{\ell-1} p_k^2\rt)$-bounded of
exponent 2 (using the constant $q=\epsilon^2<p^2/4$ in
Definition~\ref{definition pDqbounded}, so $q$ does not depend on $n$). Noting
that $\sum_{k=1}^{\ell-1} p_k^2<p+\epsilon$ and using
Definition~\ref{definition pDqbounded}, we see that the collection
$\{\xij+\widetilde\alpha_{ij}: \widetilde\alpha_{ij}\mbox{ has corresponding
probability } p_\ell < \epsilon\}$ is $\lt(p + \epsilon\rt)$-bounded of
exponent 2.

Combining Case 1 and Case 2, we have that the collection
$\{\xij+\widetilde\alpha_{ij}\}$ is $\lt(p + \epsilon\rt)$-bounded of exponent
2, and so by and by Theorem~\ref{main theorem} we have that $\Pr(X_n+\fglm
\mbox{ is singular}) 
\le \lt(\sqrt{p +\epsilon} + o(1)\rt)^n$.

The constant $\epsilon>0$ was chosen arbitrarily, and so letting
$\epsilon$ tend to zero, we get that 
$$ \Pr(X_n+\glm \mbox{ is singular} |
\{\eta_{ijk}=c_{ijk}\}) \le \lt( \sqrt p + o(1)\rt)^n.$$
\end{proof}

\subsection{Partially random matrices }
\label{ss:fxt}

In this subsection, we prove a bound on the singularity probability for partly
random matrices where many rows are deterministic.

\begin{corollary} \label{cor:partial1} 
Let $p$ be a real constant between 0 and 1, let $K$ be a large positive
constant, and let $S\subset \bb C$ be a set of complex numbers having
cardinality $\abs S \le K$.  Let $\fmat$ be an $n$ by $n$ matrix in which
$\fxt$ rows contain fixed, non-random integers between $-K$ and $K$ and where
the other rows contain entries that are independent random variables taking
values in $S$.  If $\fxt \le o(n)$, if the $\fxt$ fixed rows are linearly
independent, and if for every random entry $\alpha$, we have $\max_x \Pr(\alpha
= x) \le p$, then
$$\Pr(\fmat \mbox{ is singular})\le \lt(\sqrt p + o(1)\rt)^{n-\fxt}.$$
\end{corollary}

Corollary~\ref{cor:partial1} applies
to partly random matrices with $\fxt = o(n)$ fixed, non-random rows containing
integers bounded by a constant and with random entries taking at most $O(1)$
values in the complex numbers.   Corollary~\ref{cor:partial}, on the other
hand, holds with the fixed entries also allowed to take values in the complex
numbers and gives a sligtly better bound, but additionally requires $\fxt \le
O(\ln n)$ (which is far smaller in general than $o(n)$).  Proving
Corollary~\ref{cor:partial} requires mirroring the entire argument used to
prove the main theorem (Theorem~\ref{main theorem}) in the case where $\fxt$
rows contain fixed, non-random entires, and we discuss this argument in
Section~\ref{S:fxt}.  Proving Corollary~\ref{cor:partial1}, however, can be
done directly from Theorem~\ref{main theorem}, as we will show below.  First,
we will state a generalization of Corollary~\ref{cor:partial1}.

\begin{corollary} \label{cor:partial2} 
Let $p$ be a real constant between 0 and 1, let $K$ be a large positive
constant, and let $S\subset \bb C$ be a set of complex numbers having
cardinality $\abs S \le K$.   Let $\fmat$ be an $n$ by $n$ matrix in which
$\fxt$ rows contain fixed, non-random integers between $-K$ and $K$ and where
the other rows contain entries that are independent random variables taking
values in $S$. If $\fxt \le o(n)$, if the fixed rows have co-rank $k$, 
and if for every random entry $\alpha$, we have $\max_x \Pr(\alpha =
x) \le p$, then
$$\Pr(\fmat \mbox{ has co-rank } >k)\le \lt(\sqrt p + o(1)\rt)^{n-\fxt}.$$
\end{corollary}

To obtain Corollary \ref{cor:partial2} from Corollary \ref{cor:partial1}, find
a collection $\mc C$ of $\fxt-k$ linearly independent rows among the
deterministic rows. Replace the rest of the deterministic rows with a
collection $\mc C'$ of rows containing integer values between $-K$ and $K$
such that $\mc C'$ is linearly independent from $\mc C$. Finally, apply
Corollary \ref{cor:partial1} to the new partially random matrix whose
deterministic rows are from $\mc C \cup \mc C'$, thus proving
Corollary~\ref{cor:partial2}.

\begin{proof}[Proof of Corollary~\ref{cor:partial1}]
By reordering the rows and columns, we may write
$$\fmat = \lt(\begin{array}{c|c} 
A & B \\
\hline
C& D
\end{array}\rt),$$
where $A$ is an $\fxt$ by $\fxt$ non-random invertible matrix, $B$ is an
$\fxt$ by $n-\fxt$ non-random matrix, $C$ is an $n-\fxt$ by $\fxt$ random
matrix, and $D$ is an $n-\fxt$ by $n-\fxt$ random matrix.  Note that $\fmat$
is singular if and only if there exists a vector $\mathbf v$ such that $\fmat
\mathbf v =0$.  Let $\mathbf v_1$ be the first $\fxt$ coordinates of $\mathbf
v$ and let $\mathbf v_2$ be the remaining $n-\fxt$ coordinates.  Then $\fmat
\mathbf v = 0$ if and only if
$$\begin{cases}
A \mathbf v_1 + B \mathbf v_2 = 0, \mbox{ and } \\
C \mathbf v_1 + D \mathbf v_2 = 0.
\end{cases}
$$
Since $A$ is invertible, these two equations are satisfied if and only if
$(-CA^{-1}B + D) \mathbf v_2 = 0$, that is, if and only if the random matrix
$-CA^{-1}B + D$ is singular.  

We want to show that every entry that can appear in $-CA^{-1}B$ is an element
of $\nonS$.  By the cofactor formula for $A^{-1}$, we know that the $i,j$
entry of $A^{-1}$ is $(-1)^{i+j}(\det A_{ij})/\det A$, where $A_{ij}$ is the
$\fxt-1$ by $\fxt-1$ matrix formed by deleting the $i$-th row and $j$-th
column of $A$.  Thus, $A^{-1} = \frac1{\det A} \widetilde A$, where the $i,j$
entry of $\widetilde A$ is $(-1)^{i+j}\det A_{ij}$.  By the volume formula for
the determinant, we know that $\abs{\det A}$ is at most the product of the
lengths of the row vectors of $A$; and thus $\abs{\det A} \le n^{o(n)}$ (here
we need that $A$ has integer entries between $-K$ and $K$, where $K$ is a
constant, and that $\fxt \le o(n)$).  Similarly, we have $\abs{\det A_{ij}}\le
n^{o(n)}$.  Every entry of $\widetilde A$ is thus in $n^{o(n)}\{-1,0,1\}$, every
entry of $C$ is in $S$, and every entry of $B$ is in $O(1)\{-1,0,1\}$; thus,
every entry of $-C\widetilde A B$ is an element of $n^{o(n)}(S\cup \{-1,0,1\})$.

Conditioning on the values taken by all the entries in $C$, we have
\en{
	\Pr(\fmat\mbox{ is singular}) &= \Pr(-CA^{-1}B + D\mbox{ is singular})
	\nonumber \\
	& = \sum_{(c_{ij})} \Pr(-CA^{-1}B + D \mbox{ is singular}| C= (c_{ij}))
	\Pr(C = (c_{ij})), \label{cond-form}
}
where the sum runs over all possible matrices $(c_{ij})$ that
$C$ can produce.   Considering the entries in $C=(c_{ij})$ to be fixed (note
that $A$ and $B$ are fixed by assumption), we now need to bound 
$$\Pr(-(c_{ij})A^{-1}B + D \mbox{ is singular})
=\Pr(-(c_{ij})\widetilde AB + (\det A) D \mbox{ is singular}).$$
Note that every entry of $-(c_{ij})\widetilde AB$ is an element of $\nonS$
and that the random matrix $(\det A) D$ has entries that take values in the
fixed set $\{(\det A) s: s\in S\}$ having cardinality $O(1)$.  Thus, by
Corollary~\ref{gen asym2}, we have that 
$$\Pr(-(c_{ij})\widetilde AB + (\det A) D \mbox{ is singular}) \le (\sqrt p +
o(1))^{n-\fxt}.$$
Plugging this bound back into Equation~\eqref{cond-form} completes the proof.
\end{proof}

\subsection{Integer matrices and rational eigenvalues}\label{ss:eigen}

Let $\eta_k$ be the random variable taking the values $-k, -k+1, \ldots, k-1,
k$ each with equal probability, and let $M_n$ be the $n$ by $n$ matrix where
each entry is an independent copy of $\eta_k$.  In \cite{WongMart}, Martin and
Wong show that for any $\epsilon >0$, $$\Pr(M_n \mbox{ has a rational
eigenvalue}) \le \frac{c(n,\epsilon)}{ k^{1-\epsilon}},$$ 
where $c(n,\epsilon)$ is a constant depending on $n$ and $\epsilon$.  (One
goal in \cite{WongMart} is to study this bound as $k$ goes to $\infty$ while
$n$ is fixed, which is why $c(n,\epsilon)$ is allowed to depend on $n$.)

Below, we prove a similar result for random integer matrices with
entries between $-k$ and $k$ (with $k$ fixed), where we allow each entry to
have a different (independent) distribution and we also allow the
distributions to be very general.
\begin{corollary}\label{norat}
Fix a positive integer $k$, and let $\Mkn$ be a random integer matrix with
independent entries, each of which takes values in the set
$\{-k,-k+1,\ldots,k-1, k\}$.  
Let $c$ be a constant such that for every entry $\alpha$,  we have
$\max_{-k\le x\le k}\Pr(\alpha =x) \le c/k$.
Then 
$$\Pr(\Mkn \mbox{ has a rational eigenvalue}) \le
\lt(\frac{c}{k} + o(1)\rt)^{n/2},$$
where the $o(1)$ term goes to zero as $n$ goes to $\infty$.
\end{corollary}

\noindent
For example, in the case where each independent entry has the uniform
distribution on $\{-k,-k+1,\ldots,k-1, k\}$ (as in \cite{WongMart}), one can
set $c=1/2$ in the corollary above.

\begin{proof}
The proof given below follows the same outline as the main theorem of
\cite{WongMart}, with Corollary~\ref{cor:partial} replacing an appeal to
\cite[Lemma~1]{WongMart}.  

The characteristic polynomial for $\Mkn$ is monic with integer coefficients,
and thus the only possible rational eigenvalues are integers (by the rational
roots theorem).  Every eigenvalue of $\Mkn$ has absolute value at most $nk$
(see \cite[Lemma~4]{WongMart}); thus, the only possible integer eigenvalues are
between $-nk$ and $nk$.  

The matrix $\Mkn$ has $\lambda$ as an eigenvalue if and only if $\Mkn -
\lambda I$ is singular (where $I$ is the $n$ by $n$ identity matrix).  By
Corollary~\ref{cor:partial} (with $\fxt = 0$), we have
$$ \Pr(\Mkn - \lambda I \mbox{ is singular}) \le \lt( \sqrt{\frac{c}{k}} +
o(1)\rt)^n.$$
Using the union bound, we have
\e{
	\Pr(\Mkn \mbox{ has a rational eigenvalue}) &=
	\Pr(\Mkn -\lambda I  \mbox{ is singular, for some $\lambda \in
	\{-nk,\ldots,nk\}$})\\
	&\le \sum_{\lambda = -nk}^{nk}  \Pr(\Mkn - \lambda I \mbox{ is singular}) \\
	&\le (2nk +1) \lt( \sqrt{\frac{c}{k}} + o(1)\rt)^n\\
	&\le \lt( \frac{c}{k} + o(1)\rt)^{n/2}.
}
\end{proof}

\section{Random matrices with complex entries:
A reduction technique}
\label{reducing from thereals}

The original work on discrete random matrices in \cite{Kom, KKS, TV1, TV2} is
concerned with matrices having integer entries, which can also be viewed as
matrices with entries in $\Zmodp$ where $Q$ is a very large prime.  In this
section we show that  one can  pass from a  (random) matrix with entries in
$\bb C$ to one with entries in $\Zmodp$ where $Q$ is an arbitrarily large
prime number, all without affecting the probability that the determinant is
zero, thanks to the following lemma.

\begin{lemma}[\cite{VWnote}]\label{reduction theorem}
Let $\sov$ be a finite subset of $\bb C$.  There exist infinitely many 
primes $Q$ such that there is a ring homomorphism $\qmap:\Z[\sov] \to \Zmodp$ 
satisfying the following two properties:
\begin{enumerate}
\item[(i)]  the map $\qmap$ is injective on $\sov$, and
\item[(ii)]  for any $n$ by $n$ matrix $(s_{ij})_{1\le i,j\le n}$ with entries $s_{ij}\in S$,
we have $$\det\lt(\rule{0pt}{14pt}(s_{ij})_{1\le i,j\le n}\rt) = 0 \quad\mbox{ if and only if } \quad
\det\left(\rule{0pt}{14pt}(\qmap(s_{ij}))_{1\le i,j\le n}\right)=0.$$
\end{enumerate}
\end{lemma}

In order to apply this lemma, let us point out that the proof of
Theorem~\ref{main theorem}, which is discussed in Sections~\ref{S:pf} through
\ref{S:8}, works exclusively in $\Zmodp$; though at various points, it is
necessary to assume $Q$ is extremely large with respect to $n$ and various
constants.  For this paper, $S$ will be the set of  all  possible values taken
by the random variables $\alpha_{jk}$.  Recall that by assumption,
$\abs{\sov}\le n^{o(n)}$, so in particular, $S$ is finite.

\begin{remark}[On the size of $Q$]\label{sizeofQ}
When we apply Lemma~\ref{reduction theorem}, we will take $Q>\exp(\exp(Cn))$
for some constant $C$ in order for Freiman-type theorems such as
\cite[Theorem~6.3]{TV2} (which is restated in Theorem~\ref{Freiman} below) to
apply, and we will also choose $Q$ large enough so that the integral
approximation in Inequality~\eqref{approximation condition}
holds and so that $Q$ is large with respect to various constants.  One should
note that while $Q$ can be taken arbitrarily large with respect to $n$, we
cannot choose $Q$ so that it is arbitrarily large with respect to $\qmap(s)$
for all $s\in \sov$, where $\sov$ is the set of all
values that could appear in the given random matrix.  For example, if $\sqrt 2
\in \sov$, then the smallest positive integer representative for $\qmap(\sqrt
2)$ must be larger than $\sqrt Q$ (since $(\qmap(\sqrt 2))^2 = 2$ in
$\Zmodp$).  Finally, if we were in a situation where $\sov\subset \bb Q$, then
we could avoid using Lemma~\ref{reduction theorem} altogether by clearing
denominators to pass to $\bb Z$ and then take $Q\approx\exp(\exp(Cn))$, as is
done in \cite{TV2}.  \end{remark}

Lemma~\ref{reduction theorem} is a corollary of the main theorem of
\cite{VWnote} and its proof is  given in detail in \cite[Section~6]{VWnote}.
The paper  \cite{VWnote} also contains  further applications of the method
used to prove Lemma~\ref{reduction theorem}, for example proving a sum-product
result for the complex numbers and proving a Szemer\'edi-Trotter-type result
for the complex numbers, where the applications follow from the analogous
results for $\bb Z/Q$ where $Q$ is a prime (see \cite{BKT}).  The results
in \cite{VWnote}, including Lemma~\ref{reduction theorem}, all go through with
the complex numbers being replaced by any characteristic zero integral domain.
Thus, the results stated in Sections~\ref{s:intro}, \ref{s:genthm}, and
\ref{S:examples} above for the complex numbers $\bb C$ also all go through
with $\bb C$ replaced by any characteristic zero integral domain.  For
example, Corollary~\ref{gen asym} becomes
\begin{corollary} Let $p$ be a constant such that $0<p\le 1$ and let $D$ be a
characteristic zero integral domain.  Let $S\subset D$ have cardinality $\abs
S\le O(1)$.  If $\glm$ is an $n$ by $n$ matrix with
independent random entries, each taking values in $S$, such that for
every entry $\alpha$, we have $\max_x \Pr(\alpha = x)\le p$, then $$\Pr(\glm
\mbox{ is singular})\le (\sqrt p + o(1))^n.$$ \end{corollary}

\section{Proof of the main theorem (Theorem~\ref{main theorem})}\label{S:pf}

The proof of Theorem~\ref{main theorem} very closely follows the proof of
\cite[Theorem 1.2]{TV2}.  Our goal is to highlight
the changes that need to be made to generalize the proof in \cite{TV2}
so that it proves Theorem~\ref{main theorem}.  A reader interested in the
details of the proof of Theorem~\ref{main theorem} should read this paper
alongside of \cite{TV2}.  Throughout the proof, we will assume that $n$ is
sufficiently large, and we will allow constants hidden in the $o(\cdot)$ and
$O(\cdot)$ notation to depend on the constants $\epsilon_{-1}, \epsilon_0,
\epsilon_1, \epsilon_2, p,q,r, \mediumdimconstant, \largedimconstant,
\LOEconstant,$ and $\mconst$.  The constants $\epsilon_{-1}, \epsilon_0,
\epsilon_1, \epsilon_2$ should be considered very small, and, in fact, we will
let them tend to zero to prove the full strength of  Theorem~\ref{main
theorem}.  The constants $p,q,r, \mediumdimconstant, \largedimconstant,
\LOEconstant,$ and $\mconst$ can be thought of as absolute, except possibly
for depending on each other.

\subsection{Definitions and preliminaries}\label{ss:defpre}

Given an $n$ by $n$ matrix $\generalmatrix$ with entries $\alpha_{ij}$, we
assume that the collection of independent random variables
$\{\alpha_{ij}\}_{1\le i,j\le n}$ is \pDqbdr\ for some fixed constants
$p$, $q$, and $r$ (here, $q$ is the constant from Definition~\ref{definition pDqbounded} which is independent of $n$).  We also assume that each $\alpha_{ij}$ takes at most
$n^{o(n)}$ distinct values.  Using Lemma~\ref{reduction theorem}, we may
assume without loss of generality that each $\alpha_{ij}$ takes values in
$\Zmodp$ for some very large prime $Q$.  The entirety of the proof will take
place over the field $\Zmodp$, and so terminology such as ``linearly
independent'', ``span'', ``dimension'', ``rank'' and so forth will always be
with respect to the field $\Zmodp$.

Let $X_i:=(\alpha_{i,1},\ldots,\alpha_{i,n})$ denote the $i$-th row of
$\generalmatrix$.  We note that $\generalmatrix$ has determinant zero if and
only if there is a linear dependency among the rows $X_i$.  It has been shown
(see 
\cite[Lemma~5.1]{TV1} and also \cite{KKS}) that the dominant contribution to
the singularity probability comes from the $X_i$ spanning a hyperplane (of
dimension $n-1$).  In particular,
\en{\label{nontriv}
	\Pr\lt(\generalmatrix \mbox{ is singular}\rt) = p^{-o(n)} 
	\mathop{\sum_{V\ \mathrm{a\ non-trivial}}}_{\mathrm{hyperplane\ in}\
	(\Zmodp)^n} \Pr(A_V),
}
where $A_V$ denotes the event that $X_1,\ldots, X_n$ span $V$, and
\emph{non-trivial} means that $V$ contains the origin, $V$ is spanned by vectors in
$S^n$ (where $S$ is the set of all possible values that can occur in $N_n$), and $\Pr(X_i \in V) > 0$ for all $i$.

As in \cite{TV2}, we will divide the non-trivial hyperplanes into $n^2$
classes, since it is then sufficient to show that the sum of $\Pr(A_V)$ over
all $V$ in a particular class is at most $(p^{1/r}+\littleo 1)^n$. 

\begin{definition}[combinatorial dimension]
\label{definition combinatorial dimension} 
Let $\MC D := 
\lt\{ \frac a n: 0 \le a \le n^2, a\in \bb Z\rt\}$.
For any $d_{\pm} \in \MC D$ such that $\dpm \ge
\frac1n$, we define the
\emphasis{combinatorial Grassmannian} $\Gr(\dpm)$ to be the set of all
non-trivial hyperplanes $V$ in $\lt(\Zmodp\rt)^n$ such that 
\begin{equation}\label{combinatorial dimension}
p^{n-\dpm+1/n} < \maximum_{1\le i\le n} \probability(X_i\in V)
\le p^{n-\dpm}.
\end{equation}
For $\dpm = 0$, we define  $\Gr(0)$ to be the set of all
non-trivial hyperplanes such that
\begin{equation*}
\maximum_{1\le i\le n} \probability(X_i\in V) \le p^n.
\end{equation*}
We will refer to $\dpm$ as the \emphasis{combinatorial
dimension} of $V$.  
\end{definition}

Note that $\Gr(\dpm) = \emptyset$ for $\dpm \ge n-1+1/n$ (by
Lemma~\ref{wOdlyz}).  We will consider
hyperplanes $V$ with combinatorial dimension in three main regions: $\dpm$
small, $\dpm$ medium-sized, and $\dpm$ large.  The two lemmas and the
proposition below suffice to prove Theorem~\ref{main theorem}.

\begin{lemma}[Small combinatorial dimension, \cite{KKS}, \cite{TV1}, \cite{TV2}]
\label{smalldim}
For any $\delta >0$ we have
\[
\sum_{\dpm \in \MC D \ \mathrm{s.t.}\ p^{n-\dpm}\le \delta^n}
\quad
\sum_{V \in \Gr(\dpm)} \probability(A_V) \le n \delta^n.
\]
\end{lemma}

In proving Theorem~\ref{main theorem}, we will take $\delta = (p +
\mediumdimconstant \epsilon_0)^{1/r}$ to take care of all small $\dpm$ not
covered by Proposition~\ref{proposition medium} below.  

\begin{proof}
The reasoning here is the same as in \cite[Lemma~2.3]{TV2}, making use of fact
that $\probability(X_i \in V) \le \maximum_{1\le i \le n} \probability(X_i \in
V) \le p^{n-\dpm} \le \delta^n$.  In particular,
$$\Pr(A_V) \le \sum_{i=1}^n \Pr\lt(\{X_j\}_{1\le j\le n}\setminus\{X_i\} \mbox{
	spans } V\rt)\Pr(X_i \in V),$$
which completes the proof since the summing the right-hand side over all $V$
is at most $n \max_{i} \Pr(X_i \in V)$ (note that an instance of the
vectors $\{X_j\}_{1\le j\le n}\setminus\{X_i\}$ can span at most one
hyperplane).
\end{proof}

\begin{lemma}[Large combinatorial dimension, \cite{KKS},\cite{TV1},\cite{TV2}]\label{large dim}
We have
\[
\sum_{\dpm \in \MC D \ \mathrm{s.t.}\ \frac{\largedimconstant}{\sqrt n}\le p^{n-\dpm}}
\quad
\sum_{V \in \Gr(\dpm)} \probability(A_V) \le (p+\littleo{1})^n
\]
\end{lemma}

Here we choose the constant $\largedimconstant$ so that $\largedimconstant \ge
\LOEconstant p^{-1/n}
\sqrt{\frac{2r}q}$,
where $\LOEconstant$ is the constant from the Littlewood-Offord inequality
(see Lemma~\ref{LOE} in Appendix~\ref{apdxA}) and $q$ is the constant from
Definition~\ref{definition pDqbounded}.

\begin{proof}
Our proof is essentially the same as \cite[Lemma~2.4]{TV2}.  Fix $V \in \Gr(\dpm)$, where $\frac{\largedimconstant}{\sqrt n} \le p^{n-\dpm}$.  Let $\imax$ be an index such that $\Pr(X_\imax \in V) = \max_{1\le i\le n} \Pr(X_i \in V)$.  By assumption,
$$\Pr(X_\imax\in V) \ge p^{n-\dpm+1/n}  \ge \frac{\largedimconstant}{\sqrt n} p^{1/n} \ge \LOEconstant \sqrt{\frac{2r}{qn}}.$$
Noting that $X_\imax\in V$ if and only if $X_\imax$ is orthogonal to the normal vector for $V$, we have by Lemma~\ref{LOE} that
$$\Pr(X_\imax \in V) \le \LOEconstant\sqrt {\frac{r}{qk}},$$
where $k$ is the number of nonzero coordinates in the normal vector to $V$.  Combining the two inequalities above shows that $k\le n/2$.

Thus, we have
\e{
\sum_{\dpm \in \MC D \ \mathrm{s.t.}\ \frac{\largedimconstant}{\sqrt n}\le p^{n-\dpm}}
\quad
\sum_{V \in \Gr(\dpm)} \probability(A_V) &\le \Pr\lt(\lt\{\mbox{\parbox{2in}{there exists a vector $\mathbf v$ with at most $n/2$ nonzero coordinates such that $N_n \cdot \mathbf v = 0$}}\rt\} \rt) \\
& \le  (p+\littleo{1})^n \qquad\qquad\qquad\quad \reason{by Lemma~\ref{crank}}
}
(Lemma~\ref{crank} is a natural generalization of \cite[Section~3.1]{KKS}; see
also \cite{Komlosnet}, \cite[Lemma~5.1]{TV1}, and \cite[Lemma~14.10]{Bol}.)
%
%
\end{proof}

\begin{proposition}[Medium combinatorial dimension estimate]\label{proposition
medium}
Let $0< \epsilon_0$ be a constant much smaller than 1, and let $\dpm \in \MC D$ be such that
$\displaystyle (p + \mediumdimconstant \epsilon_0)^{n/r} <
 p^{n-\dpm} < \frac{\largedimconstant}{\squareroot{n}}$.
Then
\[
\sum_{V\in \Gr(\dpm)} \probability(A_V) \le \littleo{1}^n.
\]
\end{proposition}
\noindent
Here we choose the constant $\mediumdimconstant$ so that $\mediumdimconstant >
\lt(\frac 1{100}+ \mconst \rt)$, where $\mconst$ is some absolute
constant such that $0<\mconst <1$ (the $\frac1{100}$ here comes from $\mubar$
as defined in Section~\ref{S:4.2} below; in \cite{TV2}, it happens that the
constant $\mconst$ is also taken to be $\frac1{100}$).

To prove Theorem~\ref{main theorem}, we can simply combine
Lemma~\ref{smalldim} with $\delta = (p + \mediumdimconstant
\epsilon_0)^{1/r}$, Lemma~\ref{large dim}, and  Proposition~\ref{proposition
medium}.  Thus, proving Proposition~\ref{proposition medium} will complete the
proof of Theorem~\ref{main theorem}.  To prove Proposition~\ref{proposition
medium}, as in \cite[Proposition~2.5]{TV2}, we will separate hyperplanes $V$
of medium combinatorial dimension into two classes, which we will call
\emph{exceptional} and \emph{unexceptional} (see Definition~\ref{def unex}).  See \cite[Section~3]{TV2} for
motivation. The unexceptional case will be proved in the remainder of this
section, and the exceptional case will be proved in
Sections~\ref{S:except}, \ref{S:7}, and \ref{S:8}.

The results in \cite{TV1} and \cite{KKS} were derived using the ideas that we
will use for the unexceptional medium combinatorial dimension case.  The idea
of considering the exceptional case separately in \cite{TV2} (and using tools
from additive combinatorics in the exceptional case) is what lead to
the improvement of Inequality~\eqref{best so far}, which gives a bound of asymptotically $\pfrac34^n$, over the $.999^n$ bound in
\cite{KKS}.

\subsection{Proof of the medium combinatorial dimension}
\label{S:4.2}
 
Before defining exceptional and unexceptional hyperplanes, we will need some new notation. 
By assumption, the collection of random variables $\{\alpha_{ij}\}_{1\le i,j\le n}$ is \pDqbounded\ of exponent
$r$ with a constant $\mu=1-p$, with random variables
$\betasuper{\mu}_{ij}$ corresponding to each $\alpha_{ij}$, and with a constant $0< q \le p$ (see Definition~\ref{definition pDqbounded}).
We also define a constant slightly smaller than $\mu$, namely $\displaystyle
\mubar := \mu-\fraction{\epsilon_0}{100}$. 
We will let $Y_i := (y_{i,1}, \ldots, y_{i,n}) :=
(\betasuper{\mubar}_{i,1},\ldots,\betasuper{\mubar}_{i,n})$ denote another row
vector that corresponds to the row vector $X_i$ ($\betasuper{\mubar}_{i,j}$
comes from the definition of \pDqbdr).  Also, we will let
\begin{equation}\label{Zstar}
\Zstar_{i,k}:= 
(\overbrace{0,\ldots,0}^{\mbox{\parbox{.5in}{\footnotesize 
$k_{\mathrm{start}}-1$\\ zeroes}}},
y_{i,k_{\mathrm{start}}}, \ldots, y_{i,k_{\mathrm{end}}}, 
\overbrace{0,\ldots, 0}^{\mbox{\parbox{.5in}{\footnotesize 
$n-k_{\mathrm{end}}$\\ zeroes}}}),
\end{equation}
where $k_{\mathrm{start}}:= \floor{(k-1)\fraction{n}{r}}+1$ and 
$k_{\mathrm{end}}:= \floor{k\fraction{n}{r}}$.  The vector
$\Zstar_{i,k}$ can be thought of as the $k$-th segment of $Y_i$ (out of $r$
roughly equal segments).  Note that $Y_i$ and $\Zstar_{i,k}$ are both defined
using $\mubar:=\mu-\fraction{\epsilon_0}{100}$, not $\mu$.  Finally, let
$\epsilon_1$ be a positive constant that is small with respect to
$\epsilon_0$, $\mconst$, and $r$.

\begin{definition}[exceptional and unexceptional]\label{def unex}
Consider a hyperplane $V$ of medium combinatorial dimension (that is, $\dpm$
satisfies the condition in Proposition~\ref{proposition medium}).
We say $V$ is \emphasis{unexceptional} if there exists an $i_0$ where $1 \le
i_0 \le n$ and there exists a $k_0$ where $1 \le k_0 \le r$ such that
\begin{equation*}\label{equation unexceptional}
\max_{1\le j\le n}\{\probability(X_{j}\in V)\} < \epsilon_1
\probability(\Zstar_{i_0,k_0} \in V).
\end{equation*}

We say $V$ is \emphasis{exceptional} if for every $i$ where
$1\le i \le n$ and for every $k$ where $1 \le k \le r$ we have
\begin{equation}\label{equation exceptional}
\epsilon_1 \probability(\Zstar_{i,k} \in V) \le \max_{1\le j\le
n}\{\probability(X_{j}\in V)\}.
\end{equation}
In particular, there exists $i_{\mathrm{max}}$ such that
$\Pr(X_{i_{\mathrm{max}}}\in V)=\maximum_{1\le j \le n}\{ \probability(X_j \in
V)\}$; and so if $V$ is exceptional, then 
\begin{equation}\label{eqn exc}
\epsilon_1 \probability(\Zstar_{i_{\mathrm{max}},k} \in V) \le
\probability(X_{i_{\mathrm{max}}}\in V) \qquad \mbox{ for every } k.
\end{equation}
We  will  refer  to $X_{i_{\mathrm{max}}}$ as the \emph{exceptional row}.
\end{definition}
\noindent
Inequality~\eqref{pDqbounded} following Definition~\ref{definition pDqbounded}
can be used to give another relationship between
$\probability(\Zstar_{i_{\mathrm{max}},k} \in V)$ and
$\probability(X_{i_{\mathrm{max}}}\in V)$ that, together with
Inequality~\eqref{eqn exc}, will be of critical importance in
Section~\ref{S:7}.

Proposition~\ref{proposition medium} follows from the two lemmas below, so
long as $\epsilon_1$ is chosen suitably small with respect to $\epsilon_0$,
$\mconst$, and $r$.

\begin{lemma}[Unexceptional space estimate]\label{lemma unexceptional}
We have
\[
\sum_{V\in\Gr(\dpm):\ V\  \mathrm{is\ unexceptional}} \probability(A_V)\le
p^{-\littleo{n}} 2^n \epsilon_1^{ \mconst \epsilon_0 n/r}.
\]
\end{lemma}

\begin{lemma}[Exceptional space estimate] \label{lemma exceptional}
We have
\[
\sum_{V\in\Gr(\dpm):\ V\  \mathrm{is\ exceptional}} \probability(A_V)\le
n^{-\fraction n{2} + \littleo n}.
\]
\end{lemma}

We will prove Lemma~\ref{lemma unexceptional} in
Section~\ref{S:4.3}, and we will prove Lemma~\ref{lemma
exceptional} in Section~\ref{S:except}.

\subsection{The unexceptional medium combinatorial dimension case}
\label{S:4.3}

The general idea for the case of an unexceptional hyperplane $V$ is to replace
some of the rows $X_i$ in the matrix $\generalmatrix$ with rows that concentrate
more sharply on the subspace $V$.  In the case where the exponent $r=1$,
replacing a row $X_i$ with
$Y_i:=(\betasuper{\mubar}_{i,1},\ldots,\betasuper{\mubar}_{i,n})$ is
successful; however, in the exponent $r= 2$ case, for example, replacing the
entire row results in a bound that is off by an exponential factor.
We solve this problem by replacing $X_i$ with only
half of $Y_i$ (with the other half of the entries being zero).
This idea easily extends to any integer $r \ge 2$ and is the
motivation for defining the vectors $\Zstar_{i,k}$ to have all zeros except
for roughly $n/r$ coordinates, as is done in Equation~\eqref{Zstar}.  The
basic utility of $\Zstar_{i_0,k_0}$ (from Definition~\ref{def unex}) is that
it concentrates more sharply on the unexceptional subspace $V$ than the vector
$X_i$ for any $i$.

Let $\Zstar_{i_0,k_0}$ be the vector from the definition of unexceptional
(Definition~\ref{def unex}) such that $\probability(X_i \in V) < \epsilon_1
\probability(\Zstar_{i_0,k_0} \in V)$ for every $i$, and set
$\Ztilde:=\Zstar_{i_0,j_0}$.  Let $m$ be the closest integer to
$\fraction{\mconst \epsilon_0 n}{r}$, where $\mconst$ is a small positive
absolute constant (for example, in \cite{TV2}, $\mconst$ is taken to be
$\frac1{100}$).  Finally, let $\Ztilde_1, \ldots, \Ztilde_m$ be copies of
$\Ztilde$, independent of each other and of $X_1, \ldots, X_n$.

\begin{lemma}[see Lemma~4.4 in \cite{TV2}]\label{lem5.8}
Let $B_{V,m}$ be the event that $\Ztilde_1, \ldots, \Ztilde_m$ are linearly
independent and lie in $V$.  Then,
\[
\probability\lt(B_{V,m}\rt) \ge p^{\littleo{n}} \quantityfraction{\maximum_{1\le i\le
n} \probability(X_i \in V)}{\epsilon_1}^m
\]
\end{lemma}

\begin{proof}
The argument follows the same reasoning as \cite[Lemma~4.4]{TV2}, however,
the quantity $2^{\dpm-n}$ in \cite{TV2} should be replaced by $\maximum_{1\le
i\le n} \probability(X_i \in V)$.
Details are provided in Appendix~\ref{apdxB}.
\end{proof}

To conclude the proof of Lemma~\ref{lemma unexceptional}, we follow the 
``row-swapping'' argument at the end of \cite[Section~4]{TV2}, with the small
change of bounding $\probability(X_i \in V)$ by $\displaystyle\maximum_{1\le
i\le n} \probability(X_i \in V)$, which we use in place of the quantity
$2^{\dpm-n}$.
Details are provided in Appendix~\ref{apdxB}.

\section{Analyzing the  exceptional medium combinatorial dimension
case}\label{S:except}

The approach for exceptional $V$ in \cite{TV2} is very different from that
used in the unexceptional case or in the large or small combinatorial
dimension cases. Using some powerful tools from additive combinatorics, the
general idea is to put exceptional hyperplanes $V$ in correspondence with a
particular additive structure called a \GAP, and then to show that
the number of the particular \GAP s that arise in this way is exceedingly
small.  The key to this approach is a structure theorem---namely,
\cite[Theorem~5.3]{TV2}. In this section, we state a slightly modified
structure theorem (Theorem~\ref{structure theorem}), and then we show
how to use Theorem~\ref{structure theorem} to prove Lemma~\ref{lemma
exceptional}.  In the beginning of Section~\ref{S:7}, we outline the changes needed to
prove the the structure theorem for our current context, and in
Sections~\ref{S:7} and \ref{S:8} we provide details. 

Before stating the structure theorem, we need some definitions and notation.
A \emph{\gap} of rank $\rnk$ is a set of the form
$$P = \{v_0 + m_1 v_1 + \cdots + m_\rnk: \abs{m_i} \le M_i/2\},$$
where the \emph{basis vectors} $v_0,v_1,\ldots,v_\rnk$ are elements of a
$\Z$-module (here, $\Zmodp$) and where the \emph{dimensions} $M_1,\ldots,
M_\rnk$ are positive integers.  We say that $v_i$ has \emph{corresponding
dimension} $M_i$.  For a given element $a = v_0 + m_1 v_1 + \cdots + m_\rnk$
in $P$, we refer to $m_1,\ldots, m_\rnk$ as \emph{coefficients} for $a$.  A
\gap\ $P$ is \emph{symmetric} if $v_0=0$, and $P$ is \emph{proper} if for each
$a\in P$, the there is a unique $\rnk$-tuple $(m_1,\ldots, m_\rnk)$ with
$\abs{m_i}< M_i/2$ that gives the coefficients for $a$.  
If $P$ is proper and symmetric, we define the $P$-norm $\Pnorm{a}$ of an element $a\in P$ to be
$$ \Pnorm{a} := \lt( \sum_{i=1}^\rnk \pfrac{m_i}{M_i}^2 \rt)^{1/2}.$$
We will use the notation $m P$, where $m$ is a positive integer, to denote the
set $\{\sum_{i=1}^m x_i : x_i\in P\}$ and the notation $P^m$, where $m$ is a
positive integer, to denote the set $\{\prod_{i=1}^m x_i : x_i \in P\}$.  If
$P$ is a \gap\ of rank $\rnk$, then so is $mP$, while $P^m$, on the other
hand, is a \gap\ of rank at most $\rnk^m$.  Also note that $\abs{mP} \le m^\rnk
\abs P$ and that $\abs{P^m} \le \abs{P}^m$. 

Let $ V$ be an exceptional hyperplane of  medium combinatorial dimension in $
\Gr(\dpm)$ and let  $ \Xex= (\alpha_1,\ldots,\alpha_n)$ be the exceptional row
(here we are using $\alpha_j$ as shorthand for $\alpha_{\imax,j}$).  Let
$(\betasupermu_1,\ldots,\betasupermu_n)$ be the row of  random variables
corresponding to $\Xex$ from the definition of \pDqbounded\ of exponent $r$,
and let $b_{j,s}$ with $1\le j\le n$ and $1\le s\le \ell_j$ be the values
taken by $\betasupermu_j$ (see Equation~\eqref{betasuper} for the definition
of $\betasupermu_j$).

Given an exceptional hyperplane $V$, there exists a representation of the form 
\[
V =\{(x_1, x_2, \ldots, x_n) \in (\Zmodp)^n \ : \ 
 x_1 a_1 + x_2a_2 +
\cdots + x_na_n =0 \}
\]
for some elements $ a_1, a _2, \ldots, a_n \in \Zmodp $.  We will call $a_1,
a_2, \ldots, a_n $ the  \emph{defining coordinates} of $V$.  
Finally, let $\tilde a_j:= \bjsj a_j$.  We will
refer to $(\tilde a_1,\ldots,\tilde a_n)$ as the \emph{scaled defining
coordinates of $V$}.  Note that once $\imax$ is fixed, so are the elements
$\bjsj$.  We should also note that the choice of $\bjsj$
among $b_{j,s}$ for $1\le s \le \ell_j$ is arbitrary---since $\betasupermu_j$
takes the values $b_{j,s}$ each with probability at least $q$, any value of
$s$ will do; and so we have taken $s=1$ for convenience.

Let $\bb H$ denote the \emph{highly rational numbers}, that is, those numbers in
$\Zmodp$ of the form $a/b \pmod Q$ where $a,b$ are integers such that $\abs{a}, \abs b
\le n^{o(n)}$ and $b \ne 0$.  The highly rational
numbers were defined in \cite[Section~8]{TV2}, and we will need a small
extension for the current paper, due to the fact that we are using the scaled
defining coordinates of $V$ instead of simply the defining coordinates of $V$.
If we were to assume that $\bjsj$ was an $O(1)$ integer for all $j$ and that
every possible value taken by $\alpha_{ij}$ was an $O(1)$ integer for all
$i,j$, then we could still use the same definition of highly rational as in
\cite{TV2}. However, if there is a $\bjsj$ or an entry $\alpha_{ij}$ in the
matrix $\generalmatrix$ that ever takes an irrational value, then when we pass
to $\Zmodp$ using Lemma~\ref{reduction theorem} we have to account for values
possibly on the order of $Q$ (see Remark~\ref{sizeofQ}), and the highly
rational numbers are not sufficient for this task.  We can overcome this
difficulty by extending to the highly $T$-rational numbers, which
contain the highly rational numbers along with all the values in a structured
set $T$ (described below). 
We will now give a rigorous definition the highly $T$-rational
numbers.

Let $T$ be a \gap\ in $\Zmodp$ with rank $O(1)$ and having cardinality at most
$n^{o(n)}$.  As in the definition of \pDqbounded\ of exponent $r$
(Definition~\ref{definition pDqbounded}), we will take $\setofvalues$ to be
the \gap\ containing all possible values in $\Zmodp$ taken by the random
variables $\alpha_{ij}$ that are the entries of $\generalmatrix$; thus, by
assumption $\abs\sov \le n^{o(n)}$.  By the definition of \pDqbdr, we know
that all of the random variables $\betasupermu_{ij}$ take values in a set with
cardinality $O(1)$.  Thus, there is a symmetric \gap\ $T$ with rank $O(1)$ and
cardinality $\abs T\le n^{o(n)}$ such that $T$ contains $S$, such that $T$
contains the set $\{-1,0,1\}$, and such that $T$ contains all the values taken
by the $\betasupermu_{ij}$.  To construct $T$ from $S$, one can, for example,
add each distinct value taken by a $\betasupermu_{ij}$ as a new basis vector
$v'$ with corresponding dimension $M':= 3$ (say).

A \emph{highly $T$-rational number} $h$ is any element of $\Zmodp$ of the form
$a/b$, where $a,b \in n^{o(n)} T^{O(1)}$.  Note that therefore, the
cardinality of the \hsr\ numbers is at most $(n^{d o(n)} \abs{T})^{O(1)} =
n^{o(n)}$, where $d=O(1)$ is the rank of $T$ (here we used the fact that
$\abs{T} \le n^{o(n)}$).

\begin{theorem}[Structure Theorem]\label{structure theorem}
There is a constant $C=C(\epmo,\epsilon_0,\epsilon_1,\epsilon_2,q,r,\mu)$ such
that the following holds.  Let $V$ be an exceptional hyperplane and let
$\tilde a_1, \ldots, \tilde a_n$ be its  scaled defining coordinates (as
described above).  Then there exist integers
\e{
	1\le \rnk \le C
}
and $M_1,\ldots, M_\rnk  \ge 1$ with the volume bound
\e{
M_1 \cdots M_\rnk \le C\Pr(\Xex \in V)^{-1}
}
and nonzero elements $v_1,\ldots, v_\rnk  \in\Zmodp$ such that the following
holds
\begin{itemize}
\item
(i)  (Scaled defining coordinates lie in a progression) The
symmetric \GAP
\e{
	P:= \{m_1v_1+\cdots +m_\rnk  v_\rnk  : -M_i/2 < m_i < M_i/2 \}
}
is proper and contains all of the $\tilde a_j$.

\item
(ii)
(Bounded norm) The $\tilde a_j$ have small $P$-norm:
\e{
	\sum_{j=1}^n \Pnorm{\tilde a_j}^2 \le C.
}
\item
(iii)
(Rational $T$-commensurability)
The set $\{v_1,\ldots,v_\rnk \} \cup \{\tilde a_1,\ldots, \tilde a_n\}$ is
contained in the set 
\e{
	\left\{ h v_1: \mbox{ $h$ is highly
	$T$-rational}\right\}.
}
\end{itemize}

\end{theorem}

\noindent
Note that unlike \cite{TV2}, part (iii) above does not necessarily place
$\{v_1,\ldots,v_\rnk \} \cup \{\tilde a_1,\ldots, \tilde a_n\}$ in a simple
arithmetic progression.

We will discuss the proof of the structure theorem in Sections~\ref{S:7} and \ref{S:8}.  In the remainder of this section, we will discuss
how to use the structure theorem to prove Lemma~\ref{lemma exceptional}.  

Fix $\dpm$ of medium combinatorial dimension (see Proposition~\ref{proposition
medium}).  Using independence of the rows, we have 
\en{
	\mathop{\sum_{V\in\Gr(\dpm):}}_{V\ \mathrm{is\ exceptional}} \Pr(A_V) &\le
	\mathop{\sum_{V\in\Gr(\dpm):}}_{V\ \mathrm{is\ exceptional}} \prod_{i=1}^n
	\Pr(X_i\in V) \nonumber\\
	&\le \abs{\lt\{V\in\Gr(\dpm): V\ \mathrm{is\ exceptional}\rt\}} \cdot
	\lt(\max_{1\le i\le n} \Pr(X_i\in V) \rt)^n.\label{5star1}
}

In \cite[Section~5]{TV2}, it is shown using Theorem~\ref{structure theorem}(i)
and (ii) and Gaussian-type methods (and the fact that $\rnk$ is bounded by a constant) that 
\e{
\abs{\lt\{V\in\Gr(\dpm): V\ \mathrm{is\ exceptional}\rt\}} \le
\frac{n^{o(n)}}{Q -1}
\mathop{\sum_{\rnk, \{M_1,\ldots, M_\rnk\}}}_{\{v_1,\ldots, v_{\rnk}\}} \lt(1 +
n^{-1/2} M_1 \cdots M_\rnk\rt)^n,
}
where the sum runs over all possible values for $\rnk$, for the $M_i$, and for
$v_1,\ldots,v_{\rnk}$.  By Theorem~\ref{structure theorem}, we know that $\rnk
\le C = O(1)$ and that $M_i \le M_1M_2\cdots M_\rnk \le C \Pr(\Xex \in V)^{-1}
\le O(1/p^n)$; thus, there are at most $n^{o(n)}$ choices for $\rnk$ and the
$M_i$.  Furthermore, there are at most $Q-1$ choices for $v_1$ (since $v_1\ne
0$), and once the value for $v_1$ has been fixed, (iii) tells us that there
are at most $n^{o(n)}$ choices for $\{v_2,\ldots,v_\rnk\}$ (since
$\abs{n^{o(n)} T^{O(1)}} \le n^{o(n)}$).  Thus, the sum runs over at most
$n^{o(n)}$ terms.  (This is the point in the proof where it is essential that
$n^{o(n)} T^{O(1)}$ has cardinality $n^{o(n)}$.)

Plugging the volume bound on $M_1 \cdots M_\rnk$ into the previous displayed
inequality, we have 
\en{
\abs{\lt\{V\in\Gr(\dpm): V\ \mathrm{is\ exceptional}\rt\}} &\le
n^{o(n)}\lt(1 + n^{-\frac{1}{2}} C\Pr(\Xex\in V)^{-1}\rt)^n
\nonumber\\
&= n^{-\frac{n}{2} +o(n)} \Pr(\Xex \in V)^{-n}, \label{5star2}
}
using the fact that $\Pr(\Xex \in V) \le \frac{\largedimconstant}{\sqrt
n}$, which is a consequence of $\dpm$ being of medium combinatorial
dimension.  Plugging in Inequality~\eqref{5star2} into
Inequality~\eqref{5star1} and summing over all $\dpm$ of medium combinatorial
dimension completes the proof of Lemma~\ref{lemma exceptional} (recall that by
assumption $\max_{1\le i\le n} \Pr(X_i\in V) =\Pr(\Xex\in V)$).

\section{Hal\'asz-type arguments}\label{S:7}

The proof of the structure theorem has two main ingredients: tools from
additive combinatorics, and Hal\'asz-type arguments using discrete Fourier
analysis.    Our proof of Theorem~\ref{structure theorem} will follow the
proof of \cite[Theorem~5.2]{TV2} very closely.  We will use results about
additive combinatorics from \cite[Section~6]{TV2} directly, and we will
discuss below the extent to which the Hal\'asz-type arguments
of \cite[Section~7]{TV2} need to be modified to work for our current context.
The proof of Theorem~\ref{structure theorem} will be given in
Section~\ref{S:8} using results from the current section, \cite[Section~6]{TV2},
\cite[Section~7]{TV2}, and \cite[Section~8]{TV2}.  Our Section~\ref{S:8}
follows \cite[Section~8]{TV2} closely, with a few modifications to prove
rational $T$-commensurability instead of only rational commensurability.

In  this  section we discuss modifications to the lemmas in
\cite[Section~7]{TV2} that are needed in order to prove Theorem~\ref{structure
theorem}.

We will use $\eQ{\cdot}$ to denote the primitive character
\[
\eQ{x} :=\exp(2\pi ix/Q).
\]
Let $\iex$ be the index of the exceptional row, so for
every $1\le k \le r$ we have
\en{\label{eqnExceptCase}
\epsilon_1 \Pr(\Zex{k} \in V) \le \Pr(\Xex \in V),
}
and recall that by Definition~\ref{def unex} we have $\Pr(\Xex\in V) =
\max_i \Pr(X_i \in V)$.  Let $(\alpha_1,\ldots,\alpha_n) :=\Xex$ with the
corresponding random variables $(\bmu_1,\ldots,\bmu_n)$ from the definition of
\pDqbounded\ of exponent $r$ (see Definition~\ref{definition pDqbounded} and Equation~\eqref{betasuper}), and let
$(a_1,\ldots,a_n)$ be the defining coordinates of $V$.  Then, using the
Fourier expansion, we can compute
\en{
\Pr(\Xex \in V) &= \bb E(\mathbf{1}_{\{\Xex\in V\}})
=\bb E\lt( \frac1Q \sum_{\xi\in \Zmodp} \eQ{\sum_{j=1}^n \alpha_j a_j \xi} \rt) \nonumber \\
&\le \frac{1}{Q} \sum_{\xi\in \Zmodp} \prod_{j=1}^n \abs{\bb
E( \eQ{\alpha_j a_j \xi})} \nonumber\\
&\le 
\frac{1}{Q} \sum_{\xi\in \Zmodp} \prod_{j=1}^n \bb
E\lt( \eQ{\betasupermu_j a_j \xi}\rt)^{1/r} \nonumber\\
&=
\frac{1}{Q} \sum_{\xi\in \Zmodp} \prod_{j=1}^n 
\lt(1-\mu +\mu \sum_{s=1}^{\ell_j} p_{j,s} \cos(2\pi b_{j,s} a_j \xi/Q)\rt)^{1/r}
\label{eqnP62a}\\
&\le
\frac{1}{Q} \sum_{\xi\in \Zmodp} \prod_{j=1}^n \lt(1-\mubar +\mubar
\sum_{s=1}^{\ell_j} p_{j,s} \cos(2\pi b_{j,s} a_j \xi/Q)\rt)^{1/r} \nonumber\\
&\le \prod_{k=1}^r \Pr(\Zex{k} \in V)^{1/r},\label{eqnP62b}
}
where the last line is an application of H\"older's inequality.

Define 
\en{
	f(\xi) &:=  \prod_{j=1}^n \lt( 1-\mu + \mu \sum_{s=1}^{\ell_j} p_{j,s}
	\cos(2\pi b_{j,s} a_j \xi/Q) \rt)^{1/r},  \label{defn_f}\\
	f_j(\xi) &:= \lt( 1-\mu + \mu \sum_{s=1}^{\ell_j} p_{j,s}
	\cos(2\pi b_{j,s} a_j \xi/Q) \rt)^{1/r},\quad  \mbox{ and} \label{defn_h}\\
	g_k(\xi) &:=  \prod_{(k-1)\frac nr <j\le k\frac n r } 
	                 \lt( 1-\mubar + \mubar \sum_{s=1}^{\ell_j} p_{j,s}
	\cos(2\pi b_{j,s} a_j \xi/Q) \rt)^{1/r}, \label{defn_g}
}
where $\mubar:= \mu - \frac{\epsilon_0}{100}$, as defined in
Section~\ref{S:4.2}.  Note that $\dsp f(\xi) = \prod_{j=1}^n f_j(\xi)$.

We will need the following analog of \cite[Lemma~7.1]{TV2}:
\begin{lemma}\label{Lem7.1}
For all $\xi \in \Zmodp$, we have
\[
\prod_{j=1}^n f_j(\xi)^{r\mubar/\mu} \le \prod_{k=1}^r g_k(\xi)
\]
\end {lemma}

\begin{proof}
This inequality may be proven pointwise (for each $j$ after expanding out the
definition of $g_k$) using the convexity of the $\log$ function, just  as in
the proof of \cite[Lemma~7.1]{TV2} (see also \cite[Lemma~7.1]{TV1}.
\end{proof}

Let $\epsilon _2 $ be sufficiently small compared to $\epsilon_1$ (we  will
specify  how  small in Inequality~\eqref{ep2suffsmall} while proving
Lemma~\ref{Lem7.2}).  Following \cite{TV2}, we define the \emph{spectrum} $
\Lambda \subset \Zmodp $ of $\{\bjs{1}a_1,\ldots, \bjs{n}a_n\} =  \{\tilde
a_1, \ldots, \tilde a_n\}$ (the scaled defining coordinates of $V$) to be
\en{\label{Lamdef}
	\Lambda := \{ \xi \in \Zmodp: f(\xi) \ge \epsilon_2\}.
}
Let $\Znorm{x}$ denote the distance from $x\in\bb R $ to the nearest integer.
Using the elementary inequality $\cos(2\pi x) \le 1 - \frac1{100}\Znorm{x}^2$,
we have
\en{
	f(\xi) &  \le \exp\lt( -\frac{\mu}{100r} \sum_{j=1}^n \sum_{s=1}^{\ell_j}
	p_{j,s} \Znorm{b_{j,s} a_j \xi/Q}^2 \rt)
	\label{fxi_upper} \\
	&\le \exp\lt( -\frac{q}{50r}\sum_{j=1}^n\Znorm{\bjsj a_j
	\xi/Q}^2\rt)\nonumber
}
($\mu p_{j,1} \ge 2q$ since $\min_x\Pr(\betasupermu_j=x)\ge q$ by
Definition~\ref{definition pDqbounded}).

Thus, there is a constant $C(\epsilon_2,q,r)$ such that
\en{\label{eqn30}
	\lt(\sum_{j=1}^n\Znorm{ \tilde a_j \xi /Q}^2 \rt)^{1/2}= 
	\lt(\sum_{j=1}^n\Znorm{ \bjsj a_j \xi /Q}^2 \rt)^{1/2} &\le
	C(\epsilon_2,q,r),
}
for every $\xi \in \Lambda$.  (E.g., the constant $C(\epsilon_2,q,r):=
\lt(\frac{50r}{q} \ln \pfrac{1}{\epsilon_2}\rt)^{1/2}$ suffices.)

\begin{lemma}\label{Lem7.2}
There exists a constant $C$ depending on
$\epmo,\epsilon_0,\epsilon_1,\epsilon_2,q,r$, and $\mu$ such that
\en{\label{eqn31}
	C^{-1} Q\Pr(\Xex\in V) \le \abs{\Lambda} \le C Q\Pr(\Xex \in V).  
}
Furthermore, for every integer $k\ge 4$ we have
\en{\label{eqn32}
	\abs{k \Lambda} \le \binom{C+k -3}{k-2} C Q\Pr(\Xex \in V) .
}
\end{lemma}

\begin{proof}
Our goal is to bound $\sum_{\xi \in \Lambda} f(\xi)$ from above and below, and
then pass to bounds on $\abs{\Lambda}$ using the fact that $\epsilon_2 \le
f(\xi)\le 1$ for all $\xi \in \Lambda$.  

Note that 
\e{
	\frac1{Q}\sum_{\xi\in\Zmodp}f(\xi) &\ge \Pr(\Xex \in V)
	&\npreason{by Equation~\eqref{defn_f} and Equation~\eqref{eqnP62a}}. 
}
Also,
\e{
	\frac1{Q}\sum_{\xi\notin\Lambda}f(\xi) 
	&=\frac1{Q}\sum_{\xi\notin\Lambda}\prod_{j=1}^n f_j(\xi) 
	\quad	
	%
	%
	\le
	\quad
	\epsilon_2^{1-\mubar/\mu}\frac1{Q}\sum_{\xi\notin\Lambda} \prod_{j=1}^n
	f_j(\xi)^{\mubar/\mu} \\
	&\le \epsilon_2^{1-\mubar/\mu}\frac1{Q}\sum_{\xi\in\Zmodp} \prod_{k=1}^r
	g_k(\xi)^{1/r} &\reason{Lemma~\ref{Lem7.1}}\\
	&\le \epsilon_2^{1-\mubar/\mu}\frac1{Q}\lt(\prod_{k=1}^r\sum_{\xi\in\Zmodp} 
	g_k(\xi) \rt)^{1/r} &\reason{H\"older's inequality}\\
	&\le \epsilon_2^{1-\mubar/\mu} \pfrac{1}{\epsilon_1}\Pr(\Xex \in V)
	&\reason{by Inequality~\eqref{eqnExceptCase}}. 
}

For the lower bound, we have
\e{
	\sum_{\xi\in\Lambda} f(\xi) &= \sum_{\xi\in\Zmodp} f(\xi) -
	\sum_{\xi\notin\Lambda} f(\xi) \\
	&\ge Q \Pr(\Xex\in V) -
	\frac{\epsilon_2^{1-\mubar/\mu}}{\epsilon_1}Q \Pr(\Xex\in V) \\
	&=Q \Pr(\Xex\in V) \lt(1 - \frac{\epsilon_2^{1-\mubar/\mu}}{\epsilon_1}\rt).
}
We can choose $\epsilon_2$ sufficiently small with respect to $\epsilon_1$ and
$1-\mubar/\mu$ so that, for example, 
\en{\label{ep2suffsmall}
	1 - \frac{\epsilon_2^{1-\mubar/\mu}}{\epsilon_1} \ge \frac12.
}

For the upper bound, we have
\e{
	\sum_{\xi\in\Lambda} f(\xi) &\le\sum_{\xi\in\Zmodp} f(\xi) \\
	&\le Q \prod_{k=1}^r \Pr(\Zex{k}\in V)^{1/r} &\reason{Inequality~\eqref{eqnP62b}}\\
	&\le Q \frac1{\epsilon_1} \Pr (\Xex \in V)
	&\reason{Inequality~\eqref{eqnExceptCase}}.
}

Thus, we have shown that $\sum_{\xi\in\Lambda} f(\xi) = \Theta(Q\Pr(\Xex\in
V))$.  Since $\epsilon_2 \le f(\xi) \le 1$ for all $\xi\in\Lambda$, we have
proven Inequality~\eqref{eqn31}.

Making use of \cite[Lemma~6.4]{TV2}, we can prove Inequality~\eqref{eqn32} by
showing $\abs{4\Lambda} \le C \abs{\Lambda}$ for some constant $C$.
Using Lemma~\ref{fxi_lower} below (for which we need to assume strict
positivity of $\bb E(e(\betasupermu_j t))$---see Remark~\ref{rem:strict
positivity}), we have that there exists a constant 
$\cfour:=\cfour(\epmo,\epsilon_2)$ such that
\e{
	f(\xi) &\ge \cfour(\epmo,\epsilon_2), 
}
for every $\xi \in 4\Lambda$.
Thus,
\e{
	\abs{4\Lambda} &\le \frac{1}{\cfour(\epmo,\epsilon_2)}\sum_{\xi\in \Zmodp} f(\xi)
	\\
	&\le \pfrac{1}{\cfour(\epmo,\epsilon_2)} \frac{Q}{\epsilon_1} \Pr(\Xex\in V) = C\abs{\Lambda},
}
for some constant $C$.  This completes the proof of Lemma~\ref{Lem7.2}.
\end{proof}

We now state and prove a lemma showing that $f(\xi)$ is at least a constant
for all $\xi\in 4\Lambda$.  In \cite{TV2}, the lemma below is unnecessary
because an inequality following from \cite[Inequality~(30)]{TV2} (which
corresponds to Inequality~\eqref{eqn30}) and the triangle inequality suffices.

\begin{lemma}\label{fxi_lower} 
Let $\Lambda$ and $f$ be defined as in Equation~\eqref{Lamdef} and
Equation~\eqref{defn_f}, respectively.  If $\xi \in 4\Lambda$, then
\e{
	f(\xi) \ge 
	\lt(\epsilon_2 \epmo^{\ln(1/\epsilon_2)} \rt)^{320000}
	=:\cfour(\epmo,\epsilon_{2}).
}
Note that $\cfour(\epmo,\epsilon_{2})$ is a constant.
\end{lemma}

\begin{proof}
Note that Inequality~\eqref{fxi_upper} implies that for any $\xi' \in
\Lambda$ we have
\e{
	\lt(\sum_{j=1}^n\sum_{s=1}^{\ell_j}
	p_{j,s}\Znorm{ b_{j,s} a_j \xi' /Q}^2 \rt)^{1/2} &\le
	\lt(\frac{100r}{\mu} \ln \pfrac{1}{\epsilon_2}\rt)^{1/2}.
}
Thus, by the triangle inequality, we have for any $\xi\in 4\Lambda$ that
\en{
	\lt(\sum_{j=1}^n\sum_{s=1}^{\ell_j}
	p_{j,s}\Znorm{ b_{j,s} a_j \xi /Q}^2 \rt)^{1/2} 
	&\le 4\lt(\frac{100r}{\mu} \ln
	\pfrac{1}{\epsilon_2}\rt)^{1/2}.\label{tribd}
}

Fix $\xi\in4\Lambda$.  
Let $k_0$ be the number of indices $j$ such that
\e{
	100\mu \sum_{s=1}^{\ell_j} p_{j,s} \Znorm{b_{j,s} a_j \xi/Q}^2 > \frac12,
}
and without loss of generality, say that these indices are $j=1,2,\ldots,k_0$.
Squaring Inequality~\eqref{tribd}, we see that $\frac{k_0}{200 \mu} \le
\frac{1600r}{\mu}\ln\pfrac1{\epsilon_2}$, and so we have
\e{
	k_0\le 320000r \ln\pfrac1{\epsilon_2},
}
which is a constant.  Thus, for the vast majority of the indices $j$, namely
$j=k_0+1,k_0+2,\ldots,n$, we have
\en{\label{goodj}
	100\mu \sum_{s=1}^{\ell_j} p_{j,s} \Znorm{b_{j,s} a_j \xi/Q}^2 \le \frac12.
}
We may now compute that 
\e{
	f(\xi) &:= \prod_{j=1}^n \lt( 1-\mu + \mu \sum_{s=1}^{\ell_j} p_{j,s}
	\cos(2\pi b_{j,s} a_j \xi/Q) \rt)^{1/r}
	\\
	&\ge
	\epmo^{k_0/r}
	\prod_{j=k_0+1}^n \lt( 1-\mu + \mu \sum_{s=1}^{\ell_j} p_{j,s}
	\cos(2\pi b_{j,s} a_j \xi/Q) \rt)^{1/r}
	&\smlreason{2.3in}{since $f(\xi')\ge \epmo$ for any $\xi'$ by the
	assumption of strict positivity---see Remark~\ref{rem:strict positivity})}
	\\
	&\ge
	\epmo^{k_0/r}
	\prod_{j=k_0+1}^n \lt( 
	1-100 \mu \sum_{s=1}^{\ell_j} p_{j,s}
	\Znorm{b_{j,s} a_j \xi/Q}^2 \rt)^{1/r}
	&\smlreason{2.3in}{since $\cos(2\pi x) \ge 1-100\Znorm{x}^2$ and the factors
	are all positive by Inequality~\eqref{goodj}}
	\\
	&\ge
	\epmo^{k_0/r}
	\exp\lt(-\frac{200 \mu}{r} \sum_{j=k_0+1}^n \sum_{s=1}^{\ell_j} p_{j,s}
	\Znorm{b_{j,s} a_j \xi/Q}^2\rt) 
	&\smlreason{2.3in}{$1-x \ge e^{-2x}$ for $0\le x \le .79$}
	\\
	&\ge
	\epmo^{320000 \ln\pfrac{1}{\epsilon_2}}
	\exp\lt(-320000\ln\pfrac{1}{\epsilon_2}\rt)
	&\smlreason{2.3in}{by Inequality~\eqref{tribd}}
	\\
	&=
	\lt(\epsilon_2 \epmo^{\ln(1/\epsilon_2)} \rt)^{320000}.
}
This completes the proof.
\end{proof}

We have shown that the spectrum $\Lambda$ has small doubling, and the next step is to use this fact to show that a set containing most of the scaled defining coordinates $\tilde{a}_j$ also has small doubling.  Towards that end, we will use the $\Lambda$-norm from \cite{TV2}, which is defined as follows:  for $x\in \Zmodp$, let $\Lnorm{x}$ be defined by
$$\Lnorm{x}:= \lt(\frac1{\abs{\Lambda}^2} \sum_{\xi,\xi'\in\Lambda}
\Znorm{x(\xi-\xi')/Q}^2\rt)^{1/2}.$$
Note that $0\le\Lnorm x \le 1$ for all $x$ and that the triangle inequality holds: $\Lnorm{x+y}\le \Lnorm x + \Lnorm y$.  We also have that
\e{
 \Lnorm x &\leq ( \frac{1}{|\Lambda|^2} \sum_{\xi,\xi' \in \Lambda} \Znorm{ x  \xi / Q}^2)^{1/2}
+ ( \frac{1}{|\Lambda|^2} \sum_{\xi,\xi' \in \Lambda} \Znorm{ x \xi' / Q}^2)^{1/2}\\
&= 2 ( \frac{1}{|\Lambda|} \sum_{\xi\in\Lambda} \Znorm{ x  \xi / Q}^2)^{1/2}.
}
Thus, squaring Inequality~\eqref{eqn30} and summing over all $\xi\in\Lambda$, we have
\en{\label{eqn33}
\sum_{j=1}^n \Lnorm{\atil_j}^2 \le 4C(\epsilon_2,q,r) =: C'.
}
We will now show that the set of all $x$ with small $\Lambda$-norm, which by Inequality~\eqref{eqn33} includes most of the $\atil_j$, has small doubling.

\begin{lemma}\emph{\cite[Lemma~7.4]{TV2}}\label{lem7.4}  
There is a constant $C$ such that the following holds.
Let $A \subseteq \Zmodp$ denote the ``Bohr set'':
$$ A := \{ x \in \Zmodp: \Lnorm x < \frac{1}{100} \}.$$
Then we have
$$ C^{-1} \Pr(\Xex\in V)^{-1}\leq |A| \leq |A+A| \leq C\Pr(\Xex\in V)^{-1}.$$
\end{lemma}

The proof of Lemma~\ref{lem7.4} is the same as in \cite{TV2}, with
the small modification that $a_j$ should be replaced with $\tilde{a}_j:=
\bjsj a_j$ and the quantity $2^{\dpm -n}$ should be replaced with $\Pr(\Xex\in
V)$ (and, of course, the field $F$ in \cite{TV2} should be replaced with
$\Zmodp$).  Also, one should note that \cite[Inequality~(30)]{TV2},
\cite[Inequality~(31)]{TV2}, and \cite[Inequality~(32)]{TV2} correspond to,
respectively, Inequalities~\eqref{eqn30}, \eqref{eqn31}, and \eqref{eqn32}.

In the next section, we will complete the proof of the structure theorem using
the lemma above.

\section{Proof of the Structure Theorem (Theorem~\ref{structure
theorem})}\label{S:8}

The key to proving the structure theorem is an application of Freiman's Theorem for finite fields.  

\begin{theorem}[see Lemma~6.3 in \cite{TV2}] \label{Freiman}  For
any constant $C$ there are constants $\rnk$ and $\delta$ such that the
following holds. Let $A$ be a non-empty subset of $\Zmodp$, a finite field of
prime order $Q$, such that  $|A+A|\leq C|A|$. Then, if $Q$ is sufficiently
large depending on $|A|$, there is a symmetric generalized arithmetic
progression $P$ of rank $\rnk$ such that $A\subset P$ and $|A|/|P| \ge
\delta$.
\end{theorem}
\noindent
Note that by Lemma~\ref{reduction theorem} we can assume that $Q$ is
sufficiently large with respect to $|A|\le C\Pr(\Xex\in V)^{-1} \le C(1/p)^n$
(this follows from $V$ being of medium
combinatorial dimension).

The set $A$ from Lemma~\ref{lem7.4} satisfies $|A+A|\le C^2|A|$, where $C\le O(1)$, and also contains all but $O(1)$ of the scaled defining coordinates $\atil_j$, since $\atil_j\notin A$ implies that $\Lnorm{\atil_j} \ge 1/100$ and Inequality~\eqref{eqn33} shows that there can be at most $100 C'= O(1)$ such $\atil_j$.
By Theorem~\ref{Freiman}, there exists a symmetric \gap\ $P=
\{m_1v_1+\cdots+m_\rnk v_\rnk: \abs{m_i}< M_i/2\}$ containing $A$ and
satisfying the bounds: 
\en{\label{eqn35}
\rank(P) & = \rnk \le O(1) \mbox{ and } \\
\label{eqn36}
\abs P & \le M_1M_2\cdots M_\rnk \le O(\Pr(\Xex\in V)^{-1}).
}
The \sgap\ $P$ is close to what is needed for Theorem~\ref{structure theorem}, since it satisfies the required volume and rank bounds.  We will show below that $P$ can be altered in ways that preserve Inequalities~\eqref{eqn35} and \eqref{eqn36} (except possibly for changing the implicit constants) so that $P$ satisfies conditions (i), (ii), and (iii) of Theorem~\ref{structure theorem}.

To show Theorem~\ref{structure theorem}(i), we will first add the remaining scaled defining coordinates $\{\atil_1,\ldots,\atil_n\} \setminus P$ (i.e., those $\atil_j$ such that $\Lnorm{\atil_j} \ge 1/100$) as new basis vectors $v'_k$ with corresponding dimensions $M_k'$ equal to (say) 3.  The resulting \gap, which we will continue to call $P$ by abuse of notation, satisfies both Inequalities~\eqref{eqn35} and \eqref{eqn36}, since there are only $O(1)$ of the $\atil_j$ with $\Lnorm{\atil_j} \ge 1/100$ (by Inequality~\eqref{eqn33}).  Second, we need to ensure that $P$ is proper, for which we will use the following lemma:

\begin{lemma}[cf. Lemma~9.3 in \cite{TV2}]\label{proper-torsion} There is an absolute constant $C_0\ge 1$ such that the following holds.
Let $P$ be a symmetric progression of rank $\rnk$
in a abelian group $G$, such that every nonzero element of $G$ has order at least $\rnk^{C_0 \rnk^3} |P|$.  Then there exists a proper symmetric \gap\ $P'$ of rank at most $\rnk$ containing $P$ such that
$$ |P'| \leq \rnk^{C_0 \rnk^3} |P|.$$
Furthermore, if $P$ is not proper and $\rnk\ge 2$, then $P'$ can be chosen to have rank an most $\rnk-1$
\end{lemma}  
\noindent
One can conclude Lemma~\ref{proper-torsion} from the proof of \cite[Lemma
9.3]{TV2} (the only difference is noting that the rank can be reduced by at
least 1 if $P$ is not proper to begin with).  Note that we can always choose
$Q$ larger than $\rnk^{C_0 \rnk^3} |P| \le O\pfrac1p^n$.

Applying Lemma~\ref{proper-torsion} gives us a proper symmetric \gap, which
again we call $P$ by abuse of notation, that contains all the $\atil_j$ and
satisfies both Inequalities~\eqref{eqn35} and \eqref{eqn36}.

The next task is to show that $P$ can be further altered so to meet the
condition (ii) in Theorem~\ref{structure theorem}.  Note that there are only
$O(1)$ scaled defining coordinates $\atil_j$ such that $\Lnorm{\atil_j} \ge
1/100$, and so these $\atil_j$ contribute only a constant to the sum
$\sum_{j=1}^n \Pnorm{\atil_j}^2$.  On the other hand, for any $\atil_j$ with
$\Lnorm{\atil_j} < 1/100$, we have that $k \atil_j \in A \subset P$ for every
positive integer $k < \frac1{100 \Lnorm{\atil_j}}$.  We will exploit this
fact, and to do so will need the following notation.  Let $\Phi_P: P\to
\Z^\rnk$ be the map sending a point $m_1v_1+\cdots + m_\rnk v_\rnk$ in the
proper \gap\ $P$ to the unique $r$-tuple of coefficients
$(m_1,\ldots,m_\rnk)$.

If the representation for $\atil_j$ in $P$ is $\atil_j = m_1v_1 + \cdots +
m_\rnk v_\rnk$ and $k \atil_j$ is in $P$, we would like to be able to say that
the representation for $k \atil_j$ is $km_1 v_1 +\cdots + km_\rnk v_\rnk$;
i.e., we hope that $\Phi_P(k\atil_j)$ is equal to $k\Phi_P(\atil_j)$.  If this
were true, then we would have $\abs{ k m_i} \le M_i$ for $1\le i\le \rnk$,
which, if $k$ is large, would show that $\Pnorm{\atil_j}$ is small.  However,
at this point we may well have $\Phi_P(k \atil_j) \ne k\Phi_P(\atil_j)$.  A
priori, changing this to equality would require replacing $P$ with $kP$ and
then applying Lemma~\ref{proper-torsion} to get a proper symmetric \gap, but
since $k$ may be large, this would increase the volume of $P$ too much,
violating Inequality~\eqref{eqn36}.  Luckily, the lemma below provides a way
around this difficulty.    We will say that $P$ is \emph{($k_j,x_j$)-proper}
if $\Phi_{P}(k_j x_j) = k_j \Phi_{P}(x_j)$.

\begin{lemma}\label{kxproper}
There exists an absolute constant $C_1$ such that the following holds.  Let
$P$ be a symmetric proper \gap\ with rank $\rnk$ containing elements
$x_1,\ldots, x_m$, and let $k_1,\ldots,k_m$ be positive integers such that
$\ell_j x_j \in P$ for every $1\le \ell_j\le k_j$ and for every $j$.  Then,
there exists a proper symmetric \gap\ $P'$ of rank at most $\rnk$ such that
$P'$ contains $P$, 
\en{
\abs{P'} &\le \rnk^{C_1 \rnk^4} \abs P, \mbox{ and }\nonumber \\
\nonumber
P &\mbox{ is ($k_j,x_j$)-proper for every $j$.}
} 
Furthermore, if $r\ge 2$ and if there is some $j$ for which $P$ is not ($k_j,x_j$)-proper,
then $P'$ can be chosen to have rank at most $\rnk-1$.
\end{lemma}
\noindent
The proof of this lemma relies on an application of Lemma~\ref{proper-torsion}
to $2P$ (which contains $P$) along with the fact that if $\Lnorm{\atil_j} <
1/100$ then $k \atil_j \in  P$ for \emph{every} $1\le k < \frac1{100
\Lnorm{\atil_j}}$.

\begin{proof}
We proceed by induction on the rank $\rnk$.  For the base case, let $\rnk=1$
and consider $x_j\in P$ such that $k_jx_j\in P$.  Since $P$ has rank 1 in this
case, we have that $x_j = \Phi_P(x_j) v_1$ and $k_jx_j = \Phi_P(k_j x_i) v_1$.
Combining these two equations we have $k_j \Phi_P(x_j) v_1 = \Phi_P(k_j
x_j)v_1$, and dividing by $v_1$ (note that we may assume that $v_1\ne 0$), we
see that $k_j \Phi_P(x_j) = \Phi_P(k_j x_j)$. Thus $P$ is $(k_j,x_j)$-proper
for every $j$.

For $\rnk \ge 2$, we may assume that there is some $j_0$ such that $k_{j_0}
\Phi_P(x_{j_0}) \ne \Phi_P(k_{j_0} x_{j_0})$ (i.e., we assume that $P$ is not
$(k_{j_0},x_{j_0})$-proper).  We may assume that $P$ has the form
$\{m_1v_1+\cdots+ m_\rnk v_\rnk: \abs{m_i} < M_i/2\}$.  Let
$M:=(M_1,\ldots,M_\rnk)$, and let $(-M/2,M/2)$ denote the box
$\{(m_1,\ldots,m_\rnk): \abs{m_i} < M_i/2\}$.

Let $\kbar$ be the largest integer such that $\Phi_P(\kbar x_{j_0})= \kbar
\Phi_P(x_{j_0})$, so $1 \le \kbar < k_{j_0}$ and $\Phi_P((\kbar +1)
x_{j_0})\ne (\kbar +1) \Phi_P(x_{j_0})$.  Since $\kbar x_{j_0} \in P$ and
$x_{j_0} \in P$, we know that $\Phi_P(x_{j_0})\in (-M/2,M/2)$ and
$\Phi_P(\kbar x_{j_0})= \kbar \Phi_P(x_{j_0}) \in (-M/2,M/2)$; and thus,
$(\kbar +1)\Phi_P(x_{j_0}) \in (-M, M)$.  This shows that $2P$, which has
dimensions $2M=(2M_1,\ldots,2M_\rnk)$, is not proper, since it has two
distinct representations for $(\kbar+1)x_{j_0}$.

We can now apply Lemma~\ref{proper-torsion} to $2P$, thus finding a \psgap\
$P'$ of rank at most $\rnk-1$ containing $2P$ (which contains $P$) such that
$$\abs{P'} \le \rnk^{C_0 \rnk^3} \abs{2P} \le \rnk^{2 C_0 \rnk^3} \abs{P}.$$
Since $P'$ has rank at most $\rnk-1$, we have by induction that there exists $P''$ a \psgap\ of rank at most $\rnk-1$ containing $P'$ and such that 
$$\abs{P''}\le (\rnk-1)^{C_1(\rnk-1)^4} \abs{P'} \le \rnk^{C_1(\rnk-1)^4}
\rnk^{2 C_0 \rnk^3} \abs{P},$$
and such that $P''$ is ($k_j,x_j$)-proper for every $j$.
Choosing $C_1\ge 2C_0$ (for example) guarantees that $\rnk^{C_1(\rnk-1)^4}
\rnk^{2 C_0 \rnk^3} \le \rnk^{C_1\rnk^4}$, which completes the induction.
\end{proof}

Applying Lemma~\ref{kxproper}, we can generate a new proper symmetric \gap,
which again we will call $P$ by abuse of notation, such that $P$ contains the
$\atil_j$, satisfies Inequalities~\eqref{eqn35} and \eqref{eqn36}, and is
$(k_j,\atil_j)$-proper for every $\atil_j$ such that $\Lnorm{\atil_j} <1/100$,
where $k_j:= \ceiling{\frac1{200\Lnorm{\atil_j}}} \ge 1$.  We will now show
that such $P$ satisfies part (ii) of Theorem~\ref{structure theorem}.  For
$\atil_j$ such that $P$ is $(k_j,\atil_j)$-proper,
we have that $\abs{k_j m_i} \le M_i$
for each $1\le i\le \rnk$, and so
$$\Pnorm{\atil_j} = \sum_{i=1}^\rnk \pfrac{m_i}{M_i}^2 
\le \sum_{i=1}^\rnk \pfrac1{k_j}^2 \le \sum_{i=1}^\rnk (200\Lnorm{\atil_j})^2 = 40000 \rnk \Lnorm{\atil_j}^2.$$
Thus, part (ii) of Theorem~\ref{structure theorem} follows from
Inequality~\eqref{eqn33}, since $P$ is $(k_j,\atil_j)$-proper for all but
$O(1)$ of the $\atil_j$ .

The next step is to make further alterations to $P$ so that we can prove part (iii) of Theorem~\ref{structure theorem}.  The key property that we will use for (iii) is to have the set of vectors $\{\Phi_P(\atil_j):1\le j\le n\}$ span all of $\bb R^\rnk$, and we will use a rank reduction argument on $P$ to produce a new proper symmetric \gap\ satisfying this full rank property.  

\begin{lemma}\emph{\cite{TV2}}\label{spanning}
Let $P$ be a \psgap\ of rank $\rnk$ containing a set $B$ such that the set of vectors $\Phi_P(B)$ does not span $\bb R^\rnk$.  Then there exists a \sgap\ $P'$ containing $P$ such that 
\e{
\rank(P') &\le \rnk -1 \mbox{ and  } \\
\abs{P'} &\le \abs P.
}
\end{lemma}
\noindent
Note that the resulting $P'$ is not necessarily proper or
$(k_j,\atil_j)$-proper, even if $P$ had these properties.

\begin{proof}
We use the same proof here as appears in \cite[Section~8]{TV2}.
If $\{\Phi_P(\atil_j):1\le j\le n\}$ does not have rank $\rnk$, then it is
contained is a subspace of $\bb R^\rnk$ of dimension $\rnk-1$.  Thus, there
exists an integer vector $(\alpha_1,\ldots,\alpha_\rnk)$ with all the
$\alpha_i$ coprime such that $(\alpha_1,\ldots,\alpha_\rnk)$ is orthogonal to
every vector in $\{\Phi_P(\atil_j):1\le j\le n\}$.   Thus, for every $w \in
\Zmodp$ and and any $\atil_j = m_1v_1+\cdots+ m_\rnk v_\rnk$, we have that
$$\atil_j = m_1v_1+\cdots+ m_\rnk v_\rnk = m_1(v_1-w\alpha_1)+\cdots+ m_\rnk
(v_\rnk-w\alpha_\rnk).$$
Since not all the $\alpha_i$ are zero, we may assume that $\alpha_\rnk\ne 0$.
Setting $w= v_\rnk/\alpha_\rnk$ so that $v_\rnk-w\alpha_\rnk=0$, we see that
$P$ is contained in the \sgap $$P':= \{m_1' v_1'+\cdots + m_{\rnk-1}'
v_{\rnk-1}':
\abs{m_i'}< M_i/2\}$$
with rank $\rnk-1$, dimensions $M_1,\ldots, M_{\rnk-1}$ (which are the same as
the corresponding dimensions for $P$), and basis vectors $v_i':= v_i -
\alpha_iv_\rnk/\alpha_\rnk$.  By construction $|P'|\le |P|$.  \end{proof}

We can now run the following algorithm to create a \gap\ with all the desired
properties. As the input, we take the \gap\ $P$ that we arrived at after
applying Lemma~\ref{kxproper}, thus the input $P$ contains all the $\atil_j$,
satisfies Inequalities~\eqref{eqn35} and \eqref{eqn36}, and is
$(k_j,\atil_j)$-proper for every $\atil_j$ such that $\Lnorm{\atil_j}<1/100$;
however, we do not yet know whether $\Phi_P(\{\atil_j:1\le j\le n\})$ spans
$\bb R^\rnk$.
\begin{enumerate}
\item If $\Phi_P(\{\atil_j: 1\le j \le n\})$ spans $\bb R^\rnk$, then do nothing; otherwise apply Lemma~\ref{spanning}.
\item If $P$ is proper, then do nothing; otherwise apply Lemma~\ref{proper-torsion}.
\item If for every $\atil_j$ with $\Lnorm{\atil_j} < 1/100$ we have that $P$ is $(k_j,\atil_j)$-proper, then do nothing; otherwise apply Lemma~\ref{kxproper}.
\item If $P$ satisfies the three properties given in steps 1, 2, and 3, halt; otherwise, return to step 1.
\end{enumerate}
Each application of a lemma in the algorithm may disrupt some property that
other two lemmas preserve; however, we also know that each step in the
algorithm either does not change $P$ or reduces the rank of $P$ by at least 1.
Since the original input $P$ has rank $O(1)$, the algorithm must terminate in
$O(1)$ steps, giving us a \gap\ of rank $\rnk$ that satisfies
Inequalities~\eqref{eqn35} and \eqref{eqn36}, satisfies conditions (i) and
(ii) of Theorem~\ref{structure theorem}, and satisfies the condition that
$\Phi_P(\{\atil_j:1\le j\le n\}$ spans all of $\bb R^\rnk$.

Thus, all that is left to  prove  is part (iii), the claim of
rational $T$-commensurability.  Though we will not need it in the
current section, one should recall that Theorem~\ref{structure theorem} is
only useful when $\abs{n^{o(n)}T^{O(1)}} = n^{o(n)}$, where $T$ is the \sgap\
containing $\{-1,0,1\}$ and all 
possible values taken by the $\betasupermu_{ij}$ and the $\alpha_{ij}$ (see
Section~\ref{S:except}).

%
We say that a set $W$ \emph{economically $T$-spans} a set $U$ if each $u\in U$
can be represented as a highly $T$-rational linear combination of elements in
$W$, where each coefficient may be expressed as $a/b$ where $a,b \in
n^{o(n)}T^{O(1)}$ and where the implicit constants in the $o(\cdot)$ and
$O(\cdot)$ notation are uniform over $U$.

Comparing our definitions with those from \cite[Section~8]{TV2}, we note that
``highly rational'' means the same thing as ``highly $\{-1,0,1\}$-rational'',
and ``economically spans'' means the same thing as ``economically
$\{-1,0,1\}$-spans''.  Thus, it is clear that any highly rational number is
also highly $T$-rational for any $T$ containing $\{-1,0,1\}$, and also the
statement ``$W$ economically spans $U$'' implies ``$W$ economically $T$-spans
$U$'' for any set $T$ containing $\{-1,0,1\}$.  The remainder of this section
paraphrases (with some notational changes) the latter portion of
\cite[Section~8]{TV2}.

We know that $\Phi_P(\{\atil_j:1\le j\le n\}$ spans $\bb R^\rnk$.  Thus, there exists a subset $U \subset \{\atil_1,\ldots,\atil_n\}$ of cardinality $\rnk$ such that $\Phi_P(U)$ spans $\bb R^\rnk$.  Renumbering if necessary, we can write $U = \{\atil_1,\ldots,\atil_\rnk\}$.  It will be important later on that $U$ has cardinality $O(1)$.

The set $\{v_1,\ldots,v_\rnk\}$ of basis vectors for $P$ economically
$\{-1,0,1\}$-spans $\{\atil_1,\ldots,\atil_n\}$ by the definition of $P$ (note
that $M_i\le O(\Pr(\Xex\in V)^{-1}) \le O(p^{-n}) = n^{o(n)}$), and so by
Cramer's rule, the vectors $\Phi_P(U)$ economically $\{-1,0,1\}$-span the
standard basis vectors $\{e_1,\ldots,e_\rnk\}$ for $\bb R^\rnk$.  Applying
$\Phi_P^{-1}$ (recall that $\Phi_P$ is a bijection since $P$ is proper) shows
that $U$ economically $\{-1,0,1\}$-spans $\{v_1,\ldots,v_\rnk\}$.  

Following this paragraph, we will show that there exists a  single vector
$v_{i_0}$ where $1\le i_0\le \rnk $ such that $v_{i_0}$ economically $ T
$-spans $U$, which will show by transitivity that $v_{i_0}$ economically
$T$-spans $\{ \tilde a_1, \ldots, \tilde a_n\}$ (since $U$ economically
$T$-spans $\{ v_1, \ldots, v_\rnk  \} $ which economically $T$-spans $ \{
\tilde a_1, \ldots,\tilde a_n \} $; the relation ``economically $T$-spans'' is
transitive here since the sets $U$ and $ \{ v_1, \ldots, v_\rnk  \} $ have
cardinality $O(1)$).  

Let $s$ be the smallest integer such that there exists a subset of cardinality
$s$ of $\{v_1, \ldots, v_\rnk\}$ (by renumbering, say the set is
$\{v_1,\ldots,v_s\}$) so that for some nonzero $d \in \nonT$ and some
$c_{ij}\in \nonT$ we have
\en{\label{eqn37}
d \tilde a_i = \sum_{j=1}^s c_{ij} v_j \mbox{ for every $1\le i\le n$}.
}
Note that $d$ does not depend on $i$, and so this statement is slightly
stronger than having $\{v_1, \ldots, v_s\}$ economically $T$-span
$\{\tilde a_1, \ldots, \tilde a_n \}$.  Also, note that Equation~\eqref{eqn37}
holds (for example) with $s=\rnk$ by the definition of $P$ and since $T$
contains $\{-1,0,1\}$.

We now consider two cases:
\begin{itemize}
\item The $n  \times s $ matrix $C=(c_{ij})$ has rank 1 in $ \Zmodp $.  In
this case, $ \tilde a_{i_1}/\tilde a_{i_2}$ is \hsr\ for all $ i_1, i_2 $
(Since all  the $ c_{ij}$ are \hsr).  We know that $U$ \ess\ $\{v_1, \ldots,
v_\rnk  \} $, and so the numbers $v_{i_1}/v_{i_2}$ are  also \hsr\ (note that
it is critical here that $U$ has cardinality $O(1)$).  This means
that $v_1$ (for example) \ess\ $ \{ v_1, \ldots, v_\rnk  \} $, and so by
transitivity $v_1$ \ess\ $U$.

\item The matrix $C$ has rank at least 2.  Recall that $(a_1,\ldots,a_n)$
is the normal vector for $V$ and that $V$ is spanned by $(n-1)$ linearly
independent vectors with entries in $S$ (recall that $S$ contains all
possible values taken by the $\alpha_{ij}$).  We can scale the $j$-th
coordinate of each of these vectors by $\bjsj^{-1}$ to get a set of $n-1$
linearly independent vectors each of which is orthogonal to $\atil:=
(\atil_1,\ldots,\atil_n)$.  
Among these $(n-1)$ linearly independent vectors that are orthogonal to
$(\atil_1,\ldots,\atil_n)$, we can find at least one, say $w =
(\bjs{1}^{-1}w_1,\ldots,\bjs{n}^{-1}w_n)$ that is not orthogonal to every
column of $C$ (since $C$ has column rank at least 2).  Let $B:=\{\bjsj: 1\le
j\le n\}$, and let $\wtil:=w\prod_{b\in B} b= \lt(\wtil_1,
\ldots, \wtil_n \rt)$.  Thus $\wtil$ is orthogonal to $\atil$
and every coordinate $\wtil_i$ of $\wtil$ is an element of $T^{O(1)}$ (since $T$
contains $S$ and $B$ and $\abs B = O(1)$ by the definition of \pDqbdr). 

\begin{remark}\label{rlax}
Note that the line above is the only place in the proof where we use the
assumption from the definition of \pDqbdr\ that the $\betasupermu_{ij}$ take
values in a set with cardinality $O(1)$.  As is evidenced here, the following
weaker assumption suffices instead: say that for each $1\le i \le n$ there
exists a set $B_i$ such that $\abs{ B_i} = O(1)$ and such that
$\betasupermu_{i1},\betasupermu_{i2},\ldots,\betasupermu_{in}$ each take a
nonzero value in $B_i$ with probability at least $q$.  In fact, this weaker
assumption also replaces the assumption in the definition of \pDqbdr\ that $q
\le \min_x\Pr(\betasupermu_{ij} =x)$ for every $i,j$: It suffices for each
$\betasupermu_{ij}$ to take one value in $B_i$ with probability at least $q$,
instead of taking every value with probability at least $q$.
\end{remark}

We may now compute:
\e{
	0 &= d \atil\cdot \wtil = \sum_{i=1}^n d \atil_i \wtil_i 
	= \sum_{i=1}^n \sum_{j=1}^s c_{ij} v_j \wtil_i =
	\sum_{j=1}^s\lt( \sum_{i=1}^n c_{ij} \wtil_i\rt) v_j .
}
Since $\wtil$ is not orthogonal to every column of $C = (c_{ij})$, we can
assume (reordering if necessary), that the coefficient for $v_s$ above is
nonzero, and thus we have
$$
v_s = \frac{-1}{ \sum_{\ell=1}^n c_{\ell s} \wtil_\ell} \sum_{j=1}^{s-1}\lt(
\sum_{\ell=1}^n c_{\ell j} \wtil_\ell\rt) v_j .$$
Plugging this last equation into Equation~\eqref{eqn37}, we arrive at
$$
d\lt( \sum_{\ell=1}^n c_{\ell s} \wtil_\ell \rt)\atil_i
= \sum_{j=1}^{s-1} \lt(
c_{ij} \sum_{\ell=1}^n c_{\ell s} \wtil_\ell
-
c_{is}\sum_{\ell=1}^n c_{\ell j} \wtil_\ell
\rt)v_j.
$$
Since the coefficient for $\atil_i$ on the left is an element of $\nonT$ and
the coefficient for each $v_j$ on the right is an element of $\nonT$, we have
contradicted the minimality of $s$.  

\end{itemize}

Thus, we have completed the proof of the structure theorem
(Theorem~\ref{structure theorem}).  \hfill $\square$

\section{A generalization: $\numfixed$ rows have fixed, non-random values}
\label{S:fxt}

In this section, we will give a generalization of Theorem~\ref{main theorem}
to the case where the random matrix $\generalmatrix$ has $\fxt \le O(\ln n)$
rows that are assumed to be linearly independent and contain fixed, non-random
entries.  The proof of the generalized result is very similar to the proof of
Theorem~\ref{main theorem}, and we will sketch the main differences in the two
proofs below.

\begin{definition}[a random matrix $\fmat$ with entries in $S$]\label{def fmat}
Let $\fxt $ be an integer between 1 and $n $, let $S$ be a subset of a ring,
and let $ \fmat $ be  an $n$ by $n$ matrix defined as follows.  For $1 \le
i\le \fxt $ and $1 \le j \le n$, let the  entries $s_{ij}$ of $ \fmat $ be
fixed (non-random) elements of $S$ such that the rows $(s_{i,1},\ldots,
s_{i,n})$ for $1\le i \le \fxt$ are linearly independent.  For $ \fxt+1 \le i
\le n $ and $ 1 \le j \le n $, let the entries $ \alpha_{ij} $ of $ \fmat $ be
discrete  finite  random variables taking values in $S $.  Thus,
\e{
	\fmat:=
	\lt(
	\begin{matrix}
	s_{1,1} & s_{1,2}&\cdots & s_{1,n} \\
	\vdots &\cdots  & \cdots & \vdots \\
	s_{\fxt,1} &\cdots & \cdots & s_{\fxt,n} \\
	\alpha_{\fxt+1,1} & \alpha_{\fxt+1,2} & \cdots & \alpha_{\fxt+1,n}  \\
	\alpha_{\fxt+2,1} & \alpha_{\fxt+2,2} & \cdots & \alpha_{\fxt+2,n}  \\
	\alpha_{\fxt+3,1} & \alpha_{\fxt+3,2} & \cdots & \alpha_{\fxt+3,n}  \\
	\vdots & \vdots & \ddots & \vdots \\
	\alpha_{n,1} & \alpha_{n,2} & \cdots & \alpha_{n,n}
	\end{matrix}
	\rt)
	\begin{array}{l}
		\left.\rule{0pt}{24pt}\right\}
		\mbox{\parbox[c]{2in}{Fixed rows; assumed to be linearly independent}}
		\\[17pt]
		\left.\rule{0pt}{40pt}\right\}
		\mbox{\parbox[c]{2in}{\large Random rows}}
	\end{array}
}
\end{definition}

\begin{theorem}\label{fxt thm}
Let $p$ be a positive constant such that $0< p<1$, let $r$ be a positive
integer constant, and let $S$ be a \gap\ in the complex numbers with rank
$O(1)$ (independent of $n$) and with cardinality at most $\abs S \le
n^{o(n)}$.
Consider the matrix $\fmat$ with entries in $S$ (see Definition~\ref{def fmat}
above), where $\fxt \le \lt(\frac{r}{2 \ln(1/p)}-o(1)\rt)\ln n$.  If
the collection of random variables $\{\alpha_{jk}\}_{\fxt+1\le j \le n,1\le
k\le n}$  is \pDqbounded\ of exponent $r$, then
\[
\probability (\fmat \mbox { is singular})
\lessthanorequalto \max \lt\{ (p^{1/r} +\littleo{1})^n, 
(p +\littleo{1})^{n-\fxt}\rt\}.
\]
\end {theorem}

Note that the bound on the singularity probability of
$\fmat$ for $r \ge 2$ is the same as in Theorem~\ref{main theorem} (since for
$r \ge 2$, we have $n/r \ll n - c\ln n = n-\fxt$).  
This is a reflection of the fact that only the large dimension case uses the
randomness in all the rows simultaneously, and in that case the exponential
bound does not depend on $r$.  Generally speaking, the best known lower bounds
on the singularity probability of a discrete random matrix come from
a dependency among at most two random rows, and since $\fmat$ certainly has
more than two random rows, the upper bounds given in Theorem~\ref{fxt thm}
seem reasonable.

Theorem~\ref{fxt thm} leads to Corollary~\ref{cor:partial} by following a
conditioning argument very similar to that given in 
Section~\ref{S:gen asym}.

\subsection{Outline of the proof of Theorem~\ref{fxt thm}}

The proof of Theorem~\ref{fxt thm} follows the same lines of reasoning as that
of Theorem~\ref{main theorem}.  In this subsection, we will state the main
lemmas with the necessary modifications, and we will mention a few important
considerations when making the modifications.

Note that Equation~\eqref{nontriv}, which reduces the question of singularity
to one of the rows spanning non-trivial hyperplane of dimension $n-1$ holds in
the current context, using the same definition of $A_V$ and ``non-trivial
hyperplane'' (both are defined after Equation~\eqref{nontriv} in
Section~\ref{ss:defpre}).

\begin{definition}[combinatorial dimension with $\fxt$ fixed rows]
\label{def combdim f} 
Let $\MC D := 
\lt\{ \frac a n: 0 \le a \le n^2, a\in \bb Z\rt\}$.
For any $d_{\pm} \in \MC D$, we define the
\emphasis{combinatorial Grassmannian} $\Grf(\dpm)$ to be the set of all
non-trivial hyperplanes $V$ in $(\Zmodp)^n$ such that 
\begin{equation*}
p^{n-\dpm+1/n} < \maximum_{\fxt+1\le i\le n} \probability(X_i\in V)
\le p^{n-\dpm}.
\end{equation*}
For $\dpm = 0$, we define  $\Grf(0)$ to be the set of all
non-trivial hyperplanes such that
\begin{equation*}
\maximum_{\fxt+1\le i\le n} \probability(X_i\in V) \le p^n.
\end{equation*}
We will refer to $\dpm$ as the \emphasis{combinatorial dimension} of $V$.
\end{definition}

\begin{lemma}[Small combinatorial dimension, with $\fxt$ fixed rows]
\label{smalldimf}
For any $\delta >0$ we have
\[
\sum_{\dpm \in \MC D \ \mathrm{s.t.}\ \numbervalues^{\dpm}q^n\le \delta^n}
\quad
\sum_{V \in \Grf(\dpm)} \probability(A_V) \le (n-\fxt) \delta^n.
\]
\end{lemma}

\begin{proof}
The proof is the same as that for
Lemma~\ref{smalldim}; also see \cite{KKS}, \cite{TV1}, \cite{TV2}.
\end{proof}

\begin{lemma}[Large combinatorial dimension, with $\fxt$ fixed
rows]\label{large dimf} We have
\[
\sum_{\dpm \in \MC D \ \mathrm{s.t.}\ \frac{\largedimconstant}{n^{1/2}}\le \numbervalues^{\dpm}q^n}
\quad
\sum_{V \in \Grf(\dpm)} \probability(A_V) \le (p+\littleo{1})^{n-\fxt}
\]
\end{lemma}
\noindent
Here, $\largedimconstant$ is the same as in Lemma~\ref{large dim}.
\begin{proof}
The proof is the same as that for Lemma~\ref{large dim}, except now we appeal
to Lemma~\ref{crank} with $\fxt > 0$.  Note that we must assume $\fxt\le n/2$
in order to apply Lemma~\ref{crank}.   See also
\cite{KKS},\cite{TV1},\cite{TV2}.
\end{proof}

\begin{proposition}[Medium combinatorial dimension estimate, with $\fxt$ fixed
rows]\label{proposition mediumf}
Let $0< \epsilon_0$ be a constant much smaller than 1, and let $\dpm \in \MC
D$ be such that $\displaystyle (p + \mediumdimconstantf \epsilon_0)^{n/r} <
\numbervalues^{\dpm} q^n < \frac{\largedimconstant}{\squareroot{n}}$.
If $\fxt \le \lt(\frac{r}{2 \ln(1/p)}-o(1)\rt)\ln n$, then 
\[
\sum_{V\in \Grf(\dpm)} \probability(A_V) \le (p+o(1))^{n/r}.
\]
\end{proposition}

Here we choose the constant $\mediumdimconstantf$ so that $\mediumdimconstantf >
\lt(\mconst + \fconst + \frac 1{100}\rt)$, where $\mconst$ and $\fconst$ are
positive absolute constants (in particular, we need $\fconst$ such that
$\fxt\le \frac{\fconst \epsilon_0 n}{r}$, which is true for any positive
constant $\fconst$ since $\fxt\le O(\ln n)$).  As before, we will prove this
proposition by separating $V$ with medium combinatorial dimension into two
cases: exceptional and unexceptional, which are defined below using the
definition of $\Zstar_{i,k}$ from Equation~\eqref{Zstar} (this definition is
the same as in Definition~\ref{def unex} with the small change that $i$ and
$j$ are required to be between $\fxt+1$ and $n$ instead of between $1$ and
$n$).

\begin{definition}\label{def unexf}
Consider a hyperplane $V$ of medium combinatorial dimension (that is, $\dpm$
satisfies the condition in Proposition~\ref{proposition mediumf}).
We say $V$ is \emphasis{unexceptional} if there exists an $i_0$ where $\fxt+1
\le i_0 \le n$ and there exists a $k_0$ where $1 \le k_0 \le r$ such that
\begin{equation*}\label{equation unexceptionalf}
\max_{\fxt+1\le j\le n}\{\probability(X_{j}\in V)\} < \epsilon_1
\probability(\Zstar_{i_0,k_0} \in V).
\end{equation*}

We say $V$ is \emphasis{exceptional} if for every $i$ where
$\fxt+1\le i \le n$ and for every $k$ where $1 \le k \le r$ we have
\begin{equation}\label{equation exceptionalf}
\epsilon_1 \probability(\Zstar_{i,k} \in V) \le \max_{\fxt+1\le j\le
n}\{\probability(X_{j}\in V)\}.
\end{equation}
In particular, there exists $i_{\mathrm{max}}$ such that
$\Pr(X_{i_{\mathrm{max}}}\in V)=\maximum_{\fxt+1\le j \le n}\{ \probability(X_j \in
V)\}$; and so if $V$ is exceptional, then 
\begin{equation}\label{eqn excf}
\epsilon_1 \probability(\Zstar_{i_{\mathrm{max}},k} \in V) \le
\probability(X_{i_{\mathrm{max}}}\in V) \qquad \mbox{ for every } k.
\end{equation}
We  will  refer  to $X_{i_{\mathrm{max}}}$ as the \emph{exceptional row}.
\end{definition}

\begin{lemma}[Unexceptional space estimate, with $\fxt$ fixed
rows]\label{lemma unexceptionalf}
If $\fxt \le \frac{\fconst \epsilon_0 n }{r}$ for some positive constant
$\fconst$, then we have
\[
\sum_{V\in\Grf(\dpm): V\  \mathrm{is\ unexceptional}} \probability(A_V)\le
p^{-\littleo{n}} 2^n \epsilon_1^{\mconst \epsilon_0 n/r}.
\]
\end{lemma}

Notice that the bound is the same as in Lemma~\ref{lemma
unexceptional}, except that we replaced $\mediumdimconstant$ with
$\mediumdimconstantf$ when defining ``unexceptional''. 

\begin{proof}
The proof follows in the same way as that for Lemma~\ref{lemma
unexceptional}; 
however, when replacing rows $X_i$ of $\fmat$ with rows
$\Ztilde_i$ that concentrate more sharply on $V$, we must take care to only
replace random rows of $\fmat$ (i.e., rows $X_{1},\ldots, X_{\fxt}$ must not
be replaced by $\Ztilde_i$).  See Appendix~\ref{apdxB} for details.
\end{proof}

In the exceptional case, The same structure theorem (Theorem~\ref{structure
theorem}) holds, leading to the following lemma.

\begin{lemma}[Exceptional space estimate, with $\fxt$ fixed rows]
If $\fxt \le \lt(\frac{r}{2 \ln(1/p)}-o(1)\rt)\ln n$, then 
\en{\label{ineq exceptf}
\sum_{V\in\Gr(\dpm): V\  \mathrm{is\ exceptional}} \probability(A_V)\le
p^{n/r}
}
\end{lemma}

Note that this upper bound is dramatically worse than the analogous upper
bound in Lemma~\ref{lemma exceptional} of $n^{-\frac n{2} + o(n)}$.

\begin{proof}
As in Lemma~\ref{lemma exceptional}, the  main  step  in  the  proof  is
applying the structure theorem (Theorem~\ref{structure theorem}).  In the
current context,
Inequality~\eqref{5star1} holds with $n-\fxt$ as the exponent instead  of $n$
(since  there  are  only  $n-\fxt$ random rows). If we combine this modified
version of Inequality~\eqref{5star1} with
Inequality~\eqref{5star2}, then we have the bound
\e{
	\mathop{\sum_{V\in\Grf(\dpm):}}_{V\ \mathrm{is\ exceptional}} \Pr(A_V) &\le
	n^{-\frac{n}{2} +o(n)} \Pr(\Xex \in V)^{-n}\Pr(\Xex \in V)^{n-\fxt}\\
	&= n^{-\frac{n}{2} +o(n)} \Pr(\Xex \in V)^{-\fxt},
}
where by assumption $\Xex$ is the random row such that $\Pr(\Xex \in V) =
\max_{\fxt+1\le i\le n} \Pr(X_i\in V)$.  In order for this upper bound to
achieve the desired bound in Inequality~\eqref{ineq exceptf},  it is
sufficient to have
\en{\label{usefstar}
	n^{-\frac{n}{2} +o(n)} \Pr(\Xex \in V)^{-\fxt} \le p^{n/r}.
}
Using the assumption that $\Pr(\Xex \in V) \ge (p + \mediumdimconstantf
\epsilon_0)^{n/r}> p^{n/r}$ (since $V$ is of medium combinatorial
dimension), we see that Inequality~\eqref{usefstar} holds whenever 
$$ \fxt \le 
\lt(\frac{r}{2 \ln(1/p)}-o(1)\rt)\ln n,$$
which completes the proof.  
\end{proof}

\section*{Acknowledgments}
We would like to thank Kevin Costello for helpful conversations on the 
conditioning argument in Subsections~\ref{ss:sym} and \ref{S:gen asym}.  Also
the third author would like to thank the National Defense Science and
Engineering Fellowship and the National Science Foundation Graduate Research
Fellowship for helping fund this work.

\appendix

\section{Two background results}\label{apdxA}

\subsection{A version of the Littlewood-Offord result in $\Zmodp$}

If $S \subset \bb Q$, then we can clear denominators and prove (as in
\cite[Lemma~2.4]{TV2}) the large combinatorial dimension estimate in $\bb R$
instead of working in $\Zmodp$%
, in which case we can also use the Littlewood-Offord result over $\bb R$ (see
\cite[Corollary~7.13]{TV4}), instead of the version over $\Zmodp$ given here
in Lemma~\ref{LOE}.  When working in $\bb R$, the integral approximation of
Inequality~\eqref{approximation condition} can be replaced by a limit going to
infinity, and we do not need any extra assumptions on $Q$.  In particular, we
may take $Q\approx\exp(\exp(Cn))$ (see Remark~\ref{sizeofQ}).

For $Q$ sufficiently large with respect to $q$, $r$, and $n$, it is clear that we
have 
\en{
\label{approximation condition}
\frac1Q\sum_{\xi\in\Zmodp}
\lt(1- 2 q + 2q\cos(2\pi \xi/Q)\rt)^{k/r}
\le
\int_0^1\lt(1- 2q + 2q\cos(2\pi t)\rt)^{k/r}\, dt + \frac1n,
}
for all $1\le k\le n$.

\begin{lemma}
\label{LOE}
Let $Q$ be sufficiently large to satisfy
Inequality~\eqref{approximation condition}, and let $v_1,\ldots,
v_n\in\Zmodp$ be such that $v_1,\ldots, v_k$ are nonzero.  Let
$\{\alpha_j\}_{j=1}^n$ be a collection of random variables that are
\pDqbounded\ of exponent $r$, and let $X_{\mathbf{v}}:= \alpha_1 v_1 +\cdots
+\alpha_n v_n$.  Then, for every $x \in \Zmodp$ we have 
\e{
	\Pr(X_{\mathbf{v}} = x) \le \frac{\LOEconstant\sqrt r}{\sqrt{qk}} =
	O\pfrac{1}{\sqrt k},
}
where $\LOEconstant$ is an absolute constant.
\end{lemma}

\begin{proof}
Our proof is closely modeled on the proof of \cite[Corollary~7.13]{TV4}.  Let
$\betasupermu_j$ be the symmetric random variables from the definition of
\pDqbounded\ of exponent $r$ corresponding to $\alpha_j$ (see
Equation~\eqref{betasuper}).  Then, 
%
we can compute
\e{
\Pr(X_{\mathbf{v}} = x) &\le \frac{1}{Q} \sum_{\xi\in \Zmodp} \prod_{j=1}^k \abs{\bb
E( \eQ{\alpha_j a_j \xi})} 
	&\smlreason{1.5in}{note that $a_j = 0$ for $j > k$}\\
&\le  \prod_{j=1}^k \lt(\frac{1}{Q} \sum_{\xi\in \Zmodp} \abs{\bb
E( \eQ{\alpha_j a_j \xi})}^k \rt)^{1/k}
	&\smlreason{1.5in}{H\"older's inequality} \\
&\le \frac{1}{Q} \sum_{\xi\in \Zmodp} \abs{\bb
E( \eQ{\alpha_{j_0} a_{j_0} \xi})}^k 
	&\smlreason{1.5in}{where $j_0$ corresponds to the largest factor in the
	previous line}\\
&\le \frac1Q\sum_{\xi\in\Zmodp} \lt(1 - \mu + \mu\sum_{s=1}^{\ell_{j_0}}
p_{j_0,s}\cos(2\pi b_{j_0,s} v_{j_0} \xi/Q)\rt)^{k/r}
	&\smlreason{1.5in}{since $\alpha_{j_0}$ is \pDqbdr}\\
&\le \frac1Q\sum_{\xi\in\Zmodp} \lt(1 - 2q + 2q 
\cos(2\pi b_{j_0,1} v_{j_0} \xi/Q)\rt)^{k/r}
	&\smlreason{1.5in}{since $\mu p_{j_0,1}\ge 2 q$}\\
&= \frac1Q\sum_{\xi\in\Zmodp} \lt(1 - 2q + 2q 
\cos(2\pi \xi/Q)\rt)^{k/r}
	&\smlreason{1.5in}{by reordering the sum}.
}
Combining the above inequalities with Inequality~\eqref{approximation
condition} and following the proof of \cite[Corollary~7.13]{TV4} to bound the
integral, we have 
\e{
\Pr(X_{\mathbf{v}} = x) &\le
\int_0^1\lt(1- 2q +2q\cos(2\pi t)\rt)^{k/r}\, dt + \frac1n \\
&= \frac{\LOEconstant\sqrt r}{\sqrt{qk}} = O\pfrac1 {\sqrt k},
}
where $\LOEconstant$ is an absolute constant.
%
\end{proof}

\subsection{A generalization of a lemma due to Koml\'os~\cite{Komlosnet}}

This lemma is a generalization of the result in \cite{Komlosnet} (see also
\cite[Lemma~14.10]{Bol}, \cite[Section~3.1]{KKS}, and \cite[Lemma~5.3]{TV1}).

\begin{lemma}\label{crank}
Fix $n$, and let $p$ be a positive constants such that $0< p<1$ and
let $r$ be a positive integer constant.
Consider the matrix $\fmat$ taking values in $\Zmodp$, where $\fxt\le n/2$ and
$Q$ is large enough to satisfy Inequality~\eqref{approximation condition}.
If the collection of random entries in $\fmat$  is \pDqbounded\ of exponent
$r$, then 
\e{
	\Pr\lt(\rule{0in}{13pt}\mbox{there exists }\mathbf v \in\Omega_1 \mbox{
		such that } \fmat\cdot \mathbf v = 0\rt) 
	\le (p + o(1))^{n-\fxt},
}
where
\e{
	\Omega_1:=\lt\{ (v_1, \ldots, v_n) \in \Zmodp: \mbox{ at most }
	(n-\fxt)\lt(1-\frac{c}{\ln n}\rt)+1 \mbox{ of the } v_i \mbox{ are
	nonzero}\rt\}\setminus\{ \underline{0}\},
}
where the constant $c$ can be taken to be $c\ge 2\ln(100/p)$, and
where $\underline{0}$ denotes the zero vector.
\end{lemma}

\begin{proof}
Let $E_k = \{\mbox{there exists } v\in \Omega_1 \mbox{ with at most $k$
nonzero coordinates such that } \fmat\cdot v = 0\}$.  Clearly,
\e{
	\Pr\lt(\rule{0in}{13pt}\mbox{there exists } v\in \Omega_1 \mbox{ such that
	} \fmat\cdot v = 0\rt) 
	\le \sum_{1\le k\le (n-\fxt)\lt(1-\frac{c}{\ln n}\rt) +1}
	\Pr(E_{k}\setminus E_{k-1}).
}

Let $\sov$ be the set of all possible values that could appear as entries in
$\fmat$, and let $\fmatkcol$ be the $n$ by $k$ matrix consisting of columns
$j_1, \ldots, j_k$ of $\fmat$.  
Following \cite[Lemma 2]{Komlosnet} (see also \cite[Lemma~14.10]{Bol} and
\cite[Lemma~5.3]{TV1}) we can write
\nc\Rspan{\operatorname{RwSpn}_{i_1,\ldots,i_{k-1},H}}
\nc\Restin{\operatorname{RwIn}_{i_1,\ldots,i_{k-1},H}}
\e{
	\Pr(E_k\setminus E_{k-1}) \le
	\mathop{\sum_{1\le j_1<\cdots}}_{\quad \cdots<j_k\le n}
	\mathop{\sum_{1\le i_1<\cdots}}_{\quad \cdots<i_{k-1}\le n}
	\ \ 
	\sum_{\mbox{\parbox{.8in}{\scriptsize $H$ a $(k-1)$-dimensional
	hyperplane spanned by $\sov^k$}}}
	\Pr(\Rspan)
	\Pr(\Restin),
}
where 
\e{
	\Rspan&:=\lt\{\rule{0in}{13pt} 
	\mbox{rows $i_1,\ldots,i_{k-1}$ of $\fmatkcol$ span $H$}\rt\}, 
	\mbox{ and}\\
	\Restin&:=\lt\{\rule{0in}{13pt}  
	\mbox{all rows of $\fmatkcol$ except $i_1,\ldots,i_{k-1}$ are in $H$} \rt\}.
}

Let $U(k,p,q)$ be a uniform upper bound for $ \Pr(\mbox{row
$i$ is in $H$})$, where $\fxt+1\le i\le n$ and $q$ is the constant from Definition~\ref{definition pDqbounded} (here, we mean uniform with respect to the index sets $\{j_1,\ldots,j_k\}$ and $\{i_1,\ldots,i_k\}$).  Then we have 
\e{
	\Pr(E_k\setminus E_{k-1}) \le
	U(k,p,q)^{n-k-\fxt+1}\binom{n}{k}\binom{n}{k-1},
}
since $k-1$ fixed rows of $\fmatkcol$ can span at most 1 hyperplane $H$ of
dimension $k-1$.

For $\dsp k \le \frac{2^{8}\LOEconstant^2r}{p^{2}q}$ (a constant), we can set
$U(k,p,q)=p$ by the Weighted Odlyzko Lemma (see Lemma~\ref{wOdlyz}), giving us a bound of 
\en{\label{7bd}
\Pr(E_k\setminus E_{k-1}) \le (p+o(1))^{n-\fxt}.
}

For $\frac{2^{8}\LOEconstant^2r}{p^{2}q}< \dsp k \le (n-\fxt)\lt(1-\frac c{\ln
n}\rt)+1$, we use Lemma~\ref{LOE} to set
$U(k,p,q)=\frac{\LOEconstant\sqrt r}{\sqrt{qk}}$.  Since 
$\binom{n}{k}\binom{n}{k-1}\le \frac{2^{2n}}{n}$ we thus have

\e{
	\Pr(E_k\setminus E_{k-1}) \le\frac1n 2^{2n}
	\pfrac{\LOEconstant^2r}{qk}^{\frac{n-k-\fxt+1}{2}}.
}
As a function of $k$, this upper bound has strictly positive second
derivative; thus, the largest upper bound will occur at one of the extremal
values of $k=\frac{2^{8}\LOEconstant^2r}{p^{2}q}$ or $k=(n-\fxt)\lt(1-\frac
c{\ln n}\rt)+1$, and a bit of computation shows that 
\en{\label{8bd}
	\Pr(E_k\setminus E_{k-1}) \le\frac1n O(p^{n-\fxt}).
}

Summing the bounds in Inequalities~\eqref{7bd} and \eqref{8bd} completes the
proof.
\end{proof}

\section{The unexceptional case with $\fxt$ fixed rows}\label{apdxB}

This section is adapted from the proof of \cite[Lemma~4.1]{TV2}, and proves
Lemma~\ref{lemma unexceptional} by setting $\fxt = 0$.  
Assume that $\fxt \le
\frac{\fconst \epsilon_0 n}{r}$, and let $m$ be the closest integer to
$\frac{\mconst \epsilon n}{r}$.  Let
$Z_1,\ldots, Z_m$ be i.i.d.\ copies of the unexceptional row vector
$\Zstar_{i_0,k_0}$ from Definition~\ref{def unexf},
so $\epsilon_1 \Pr(Z_i \in V) > \Pr(X_i \in V)$ for all $\fxt+1\le i\le n$.
We will need the following version of the Weighted Odlyzko Lemma:

\begin{lemma}\emph{[cf. \cite[Lemma~4.3]{TV2} or \cite[Section~3.2]{KKS}]}
\label{wOdlyz} For $1\le i$, let $W_{i-1}$ be an ($\fxt+i-1$)-dimensional
subspace containing $X_1,\ldots,X_\fxt$ (which are fixed, linearly independent
row vectors).  Then $$\Pr(Z_i \in W_{i-1}) \le \lt(p+
\frac{\epsilon_0}{100}\rt)^{\frac n r -\fxt-i+1}.$$
\end{lemma}

\begin{proof}
Since $W_{i-1}$ has dimension $\fxt +i -1$, there exists a set of $\fxt+i-1$
``determining'' coordinates such that if a vector $V\in W_{i-1}$, then the
$\fxt+i-1$ ``determining'' coordinates determine the values of the remaining
$n-\fxt-i+1$ coordinates.  Since the maximum probability that any of the $n/r$
random coordinates in $Z_i$ takes a given value is at most $1-\mubar =
p+\frac{\epsilon_0}{100}$, and since there are at least $\frac n r
-\fxt-i+1$ of the random coordinates in $Z_i$ that are not among the
``determining'' coordinates, we have the desired upper bound.
\end{proof}

Let $V_0 := \operatorname{Span}\{X_1,\ldots,X_\fxt\}$, the space spanned by
the $\fxt$ fixed rows, and for $1\le i\le m$ let 
$B_{V,i}$ be the event that $Z_1,\ldots,Z_m$ are linearly independent in
$V\setminus V_0$.  We have
the following analog of Lemma~\ref{lem5.8} (and also \cite[Lemma~4.4]{TV2}):

\begin{lemma}[see Lemma~4.4 in \cite{TV2}]\label{lem5.8f}
Let $m$, $\fxt$, and $B_{V,m}$ be as defined above.  Then,
\[
\probability\lt(B_{V,m}\rt) \ge p^{\littleo{n}}
\quantityfraction{\maximum_{\fxt+1\le i\le
n} \probability(X_i \in V)}{\epsilon_1}^m
\]
\end{lemma}

\begin{proof}
Using Bayes' Identity, we have
\en{\label{bayesf}
	\Pr(B_{V,m}) = \prod_{i=1}^m \Pr(B_{V,i}| B_{V,i-1}),
}
where $B_{V,0}$ denotes the full space of the $Z_i$.
Conditioning on a particular instance of $Z_1,\ldots, Z_{i-1}$ in $B_{V,i-1}$,
we have that 
$$\Pr(B_{V,i}| B_{V,i-1}) = \Pr(Z_i\in V) - \Pr(Z_i \in W_{i-1}),$$
where $W_{i-1}$ denotes the ($\fxt+i-1$)-dimensional space spanned by
$X_1,\ldots, X_\fxt$ and $Z_1,\ldots, Z_{i-1}$.  We will now establish a
uniform bound that does not depend on which particular
instance of $Z_1,\ldots, Z_{i-1}$ in $B_{V,i-1}$ that we fixed by
conditioning.  By the definition of unexceptional, we have
$$\Pr(Z_i\in V) > \frac1{\epsilon_1} \max_{\fxt+1\le i\le n}\Pr(X_i\in V),$$ and by the Weighted
Odlyzko Lemma (see Lemma~\ref{wOdlyz}), we have 
$$\Pr(Z_i \in W_{i-1}) 
\le \lt(p+\frac{\epsilon_0}{100}\rt)^{\frac n r -\fxt -i +1} 
\le \lt(p+\frac{\epsilon_0}{100}\rt)^{\frac n r (1- (\mconst+\fconst)
\epsilon_0)}.$$ Using Taylor's Theorem with remainder (for example), one can
show that 
$$\lt(p+\frac{\epsilon_0}{100}\rt)^{\frac n r (1- (\mconst+\fconst) \epsilon_0)} \le
\frac1{2n}\lt(p + \mediumdimconstant \epsilon_0\rt)^{n/r} \le \frac 1 n
\max_{\fxt+1\le i\le n} \Pr(X_i\in V),$$ so long as $\mediumdimconstant >
\frac1{100} + \mconst +\fconst> \frac 1{100} + (\mconst+\fconst) p\ln\pfrac1p$
and $n$ is sufficiently large (the second inequality in the display
above is the definition of medium combinatorial dimension).

Thus $$\Pr(B_{V,i}| B_{V,i-1}) \ge \frac1{\epsilon_1}\lt(\max_{\fxt+1\le i\le
n} \Pr(X_i\in V)\rt)\lt(1 - \frac{\epsilon_1}{ n} \rt),$$
and plugging this estimate back into Inequality~\eqref{bayesf} we get
$$\probability\lt(B_{V,m}\rt) \ge p^{\littleo{n}}
\quantityfraction{\maximum_{\fxt+1\le i\le
n} \probability(X_i \in V)}{\epsilon_1}^m.$$
\end{proof}

To conclude Lemma~\ref{lemma unexceptionalf} (which implies Lemma~\ref{lemma
unexceptional} by setting $\fxt=0$), we will proceed as in the proof for
\cite[Lemma~4.1]{TV2}.

Let $Z_1,\ldots,Z_m$ be i.i.d.\ copies of $\Zstar_{i_0,k_0}$ that are 
independent of the random rows $X_{\fxt+1},\ldots , X_n$.  Using independence
and Bayes' Identity we have
$$\Pr(A_V) = \Pr(A_V | B_{V,m}) = \frac{\Pr(A_V \wedge
B_{V,m})}{\Pr(B_{V,m})} \le 
\Pr(A_V \wedge B_{V,m})
p^{-o(n)}\quantityfraction{\epsilon_1}{\maximum_{\fxt+1\le i\le
n} \probability(X_i \in V)}^m
.$$
Because the $Z_i$ are linearly independent in $V\setminus V_0$, we know that
there is a
subset $I \subset \{\fxt+1,\fxt+2,\ldots, n\}$ of cardinality $\abs I =
m$, such that $\{Z_1,\ldots, Z_m\}\cup \{X_i: i \notin I\}$ spans $V$.
Let $C_{V,I}$ be the event that $\{Z_1,\ldots, Z_m\}\cup \{X_i: i \notin I\}$
spans $V$.  Then we have 
\e{
\Pr(A_V \wedge B_{V,m}) &\le \mathop{\sum_{I\subset \{\fxt+1,\ldots,
n\}}}_{\abs I = m} \Pr\lt(C_{V,I} \wedge \{X_i \in V : i\in I\}\rt)\\
&\le  \lt(\maximum_{\fxt+1\le i\le n} \probability(X_i \in
V)\rt)^m \mathop{\sum_{I\subset \{\fxt+1,\ldots,
n\}}}_{\abs I = m} \Pr(C_{V,I}) .
}
Summing the above inequality over all unexceptional $V$ (note that $\sum_V
\Pr(C_{V,I}) \le 1$) and combining with the bound for $\Pr(A_V)$ above gives us
\e{
	\sum_{\mathrm{unexceptional}\ V} \Pr(A_V) & \le 
	\lt(\maximum_{\fxt+1\le i\le n} \probability(X_i \in V)\rt)^m
	\binom{n-\fxt}{m}
	p^{-o(n)}\quantityfraction{\epsilon_1}{\maximum_{\fxt+1\le i\le n}
	\probability(X_i \in V)}^m\\
&\le p^{-o(n)}2^n\epsilon_1^m.
}
This completes the proof of the estimate for unexceptional $V$.


\end{document}